\documentclass[12pt,dvipsnames]{amsart}
\usepackage[colorlinks=true, pdfstartview=FitV, linkcolor=blue, citecolor=blue, urlcolor=blue]{hyperref}
\usepackage{geometry,pb-diagram}
\usepackage{amssymb, amscd, amsmath, amsfonts, amsthm}
\usepackage{stmaryrd}
\usepackage{verbatim}
\usepackage{multicol}

\newcommand{\idiot}[1]{\vspace{5 mm}\par \noindent
\marginpar{\textsc{For longer version}}
\framebox{\begin{minipage}[c]{.99 \textwidth}
#1 \end{minipage}}\vspace{5 mm}\par}

\newcommand{\rant}[1]{\vspace{5 mm}\par \noindent
\marginpar{\textsc{Comments}}
\framebox{\begin{minipage}[c]{.99 \textwidth}
\tt #1 \end{minipage}}\vspace{5 mm}\par}

\renewcommand{\idiot}[1]{}
\renewcommand{\rant}[1]{}

\def\unprotectedboldentry#1{\textcolor{Red}{\textbf{#1}}}
\def\boldentry{\protect\unprotectedboldentry}
\usepackage{tikz}
\usetikzlibrary{calc}
\usepackage{ifthen}
\newcommand{\tikztableau}[2][scale=0.6,every node/.style={font=\small}]{
    \def\newtableau{#2}
    \begin{array}{c}
    \begin{tikzpicture}[#1]
    \coordinate (x) at (-0.5,0.5);
    \coordinate (y) at (-0.5,0.5);
    \foreach \row in \newtableau {
        \coordinate (x) at ($(x)-(0,1)$);
        \coordinate (y) at (x);
        \foreach \entry in \row {
            \ifthenelse{\equal{\entry}{X}}
               {
                \node (y) at ($(y) + (1,0)$) {};
                \fill[color=gray!10] ($(y)-(0.5,0.5)$) rectangle +(1,1);
                \draw[color=gray] ($(y)-(0.5,0.5)$) rectangle +(1,1);
               }
               {
                \ifthenelse{\equal{\entry}{\boldentry X}}
                   {
                    \node (y) at ($(y) + (1,0)$) {};
                    \fill[color=gray] ($(y)-(0.5,0.5)$) rectangle +(1,1);
                    \draw ($(y)-(0.5,0.5)$) rectangle +(1,1);
                   }
                   {
                    \node (y) at ($(y) + (1,0)$) {\entry};
                    \draw ($(y)-(0.5,0.5)$) rectangle +(1,1);
                   }
               }
            }
        }
    \end{tikzpicture}
    \end{array}}
\newcommand{\tikztableausmall}[1]{\tikztableau[scale=0.45,every node/.style={font=\rm\small}]{#1}}

\def\ZZ{\mathbb{Z}}
\def\NN{\mathbb{N}}
\def\BB{\mathbb{B}}
\def\QQ{\mathcal{Q}}
\def\B{{\bf B}}

\def\content{\operatorname{content}}

\def\sym{\operatorname{\mathsf{Sym}}}

\def\Qsym{\operatorname{\mathsf{QSym}}}

\def \fS{{\mathfrak S}}
\def \HH{{H}}
\def\Nsym{\operatorname{\mathsf{NSym}}}
\def\OM{{\tilde \Omega}}
\def\coeff{{\Big|}}
\def\EEE{{\mathcal E}}
\def\HHH{{\mathcal H}}
\def\ccc{{\bf c}}
\newcommand\binomial[2]{\begin{pmatrix}#1\\#2\end{pmatrix}}
\def\sort{\operatorname{sort}}
\def\true{\operatorname{true}}
\def\false{\operatorname{false}}

\newtheorem{Theorem}{Theorem}[section]
\newtheorem{Proposition}[Theorem]{Proposition}
\newtheorem{Corollary}[Theorem]{Corollary}
\newtheorem{Lemma}[Theorem]{Lemma}
\newtheorem{Example}[Theorem]{Example}
\theoremstyle{definition}
\newtheorem{Remark}[Theorem]{Remark}
\newtheorem{Definition}[Theorem]{Definition}
\newtheorem{Conjecture}[Theorem]{Conjecture}

\begin{document}

\title[New bases for the non-commutative symmetric functions]{A lift of the Schur and Hall-Littlewood bases to non-commutative symmetric functions}
\author[C. Berg \and N. Bergeron \and F. Saliola \and L. Serrano \and M. Zabrocki]{Chris Berg$^2$ \and Nantel Bergeron$^{1,3}$ \and Franco Saliola$^2$ \and Luis Serrano$^2$ \and Mike Zabrocki$^{1,3}$}
\address[1]{Fields Institute\\ Toronto, ON, Canada}
\address[2]{Universit\'e du Qu\'ebec \`a Montr\'eal, Montr\'eal, QC, Canada}
\address[3]{York University\\ Toronto, ON, Canada}

\date{\today}

\begin{abstract} We introduce a new basis of the algebra of non-commutative symmetric functions whose images under the forgetful map  are Schur functions when indexed by a partition. Dually, we build a basis of the quasi-symmetric functions which expand positively in the fundamental quasi-symmetric functions.
We then use the basis to construct a non-commutative lift of the Hall-Littlewood symmetric functions with similar properties to their commutative counterparts.
\end{abstract}

\maketitle
\setcounter{tocdepth}{3}
\tableofcontents

\section{Introduction - Yet another Schur-like basis of $\Nsym$}

The algebras of non-commutative symmetric functions $\Nsym$ and quasi-symmetric
functions $\Qsym$ are dual Hopf algebras. They have been of great importance to
algebraic combinatorics.
As seen in \cite{ABS}, they are universal in the category of combinatorial Hopf algebras.
They also represent the Grothendieck rings for the finitely generated
projective representations and the finite dimensional representation theory of
the $0$-Hecke algebra. We will not attempt to summarize these notions in great
detail; the interested reader should see \cite{KT}.

This paper is the result of an exploration of what a `Schur' analogue in the
non-commutative setting should look like. We focused on a minimal set of axioms
that define a basis of $\Nsym$ whose image under the forgetful map is a Schur
function when the basis element is indexed by a partition.

The primary goal in this paper is to build a new basis, the \textit{``immaculate basis,''}
of $\Nsym$ and to develop its theory.
This basis has many of the same properties of the classical basis of Schur functions of the symmetric function algebra. 
The immaculate basis has a positive right-Pieri rule (Theorem \ref{thm:Pieri}), a simple Jacobi-Trudi formula (Theorem \ref{thm:JT}), and a creation operator construction (Definition \ref{def:immaculate}). Furthermore, under the forgetful map $\chi$ from the non-commutative symmetric functions to the symmetric functions (Equation \eqref{def:chi}) their image is a Jacobi-Trudi expression for a Schur function (Corollary \ref{thm:projection}).
Thus, immaculate functions map to Schur functions using a {\sl signed sorting} action (Proposition \ref{prop:schurcomposition}).
By duality, these functions give rise to a basis of the quasi-symmetric function algebra and expand positively into the monomial and fundamental bases of quasi-symmetric functions (Propositions \ref{prop:MonomialPositive} and \ref{prop:FundamentalPositive}).  We also give a combinatorial expansion of the Schur functions (Theorem \ref{thm:decompose}) in this basis. 

In the commutative setting there are many characterizing properties of the Schur functions (e.g. triangularity
with other bases, a Pieri rule, uniquely determined by an orthonormalization process). 
We immediately focused on the concept of a Jacobi-Trudi identity defining a basis of $\Nsym$.
The reader interested only in the definition of our basis may wish to take the Jacobi-Trudi rule 
(Theorem \ref{thm:JT}) as the definition of the immaculate basis. However, to simplify the proofs and provide a coherent story, we started with the functions being defined through creation operators. Since they are defined through the process of creation (applying certain creation operators to the identity), we decided to name them the ``immaculately conceived basis,'' which we have shortened to the immaculate basis. 

Recently, many of $\Nsym$'s enthusiasts have developed bases for the algebra which have various properties in common with different  classical bases of $\sym$. The ``immaculate'' term within this paper is intended to be humorous; our basis shares many of the properties of the Schur basis of $\sym$, but is in no way a perfect analogue of the Schur basis (for instance, products of immaculate functions do not expand positively in the immaculate basis). We have no supportive evidence to believe that any basis of $\Nsym$ will ever be a perfect analogue of the Schur basis.

It should be noted that this basis is not the non-commutative Schur basis 
(dual to the quasi-symmetric Schur basis of \cite{HLMvW11a, HLMvW11b, BLvW}), 
even though they share several similar properties.  The non-commutative Schur basis 
has the property that the image under the map $\chi$ of an element indexed by a composition is 
a Schur function indexed by the parts of the composition sorted in decreasing order.  The basis studied here has
the property that an element indexed by a composition is sent to a Jacobi-Trudi  determinant expression (Theorem \ref{thm:JT} and
Corollary \ref{thm:projection}).
We have not developed any connections at this point between the non-commutative Schur basis and our immaculate basis.

Starting with the immaculate basis, we construct lifts of the Hall-Littlewood symmetric functions in the non-commutative symmetric function algebra. It should be noted that many different versions of a non-commutative Hall-Littlewood symmetric function already exist \cite{BergeronZabrocki, Hivert, LascouxNovelliThibon, NovelliThibonWilliams, Tevlin}, but as far as we are aware, none of these project for all partitions to the classical Hall-Littlewood functions (ours do!). Having a basis of $\Nsym$ which projects to the Hall-Littlewood basis could prove to be a powerful tool towards their study; fundamental problems in the classical theory of Hall-Littlewood symmetric functions, like a combinatorial understanding of their structure coefficients, remain open.

The results of this paper are mostly combinatorial and algebraic. In a future
paper we turn our attention to the representation theoretic interpretation of
our basis \cite{BBSSZ1}. We construct indecomposable modules of the $0$-Hecke
algebra whose characteristics, under the identification with elements of
$\Qsym$, form the dual basis to the immaculate basis.

\subsection{Acknowledgments}
This work is supported in part by NSERC.
It is partially the result of a working session at the Algebraic
Combinatorics Seminar at the Fields Institute with the active
participation of C.~Benedetti, Z.~Chen, T.~Denton, H.~Heglin, and D.~Mazur. 

This research was facilitated by computer exploration using the open-source
mathematical software \texttt{Sage}~\cite{sage} and its algebraic
combinatorics features developed by the \texttt{Sage-Combinat}
community~\cite{sage-co}.

In addition, the authors would like to thank Florent Hivert, Jia Huang, Jeff Remmel,
Nicolas Thi\'ery and Martha Yip for lively discussions, and Darij Grinberg for several interesting mathematical and expository remarks.

The first and fourth authors are also grateful to Christophe Hohlweg for introducing us to the non-commutative symmetric functions at a working seminar at LaCIM in Montr\'eal.

The authors would like to thank the reviewers for their helpful comments and suggestions.
\section{Background}

\subsection{Compositions and combinatorics} \label{sec:compositions}

A \textit{partition} of a non-negative integer $n$ is a tuple
$\lambda = [\lambda_1, \lambda_2, \dots, \lambda_m]$ of positive integers satisfying 
$\lambda_1 \geq \lambda_2 \geq \cdots \geq \lambda_m$ which sum to $n$; it is denoted 
$\lambda \vdash n$. Partitions are of particular importance to 
algebraic combinatorics; among other things, partitions of $n$ index a basis 
for the symmetric functions of degree $n$, $\sym_n$, and the character ring 
for the representations of the symmetric group. These concepts are intimately 
connected; we assume the reader is well versed in this area (see for instance \cite{Sagan} for background details).

A \textit{composition} of a non-negative integer $n$ is a tuple 
$\alpha = [\alpha_1, \alpha_2, \dots, \alpha_m]$ of positive 
integers which sum to $n$, often written $\alpha \models n$.
The entries $\alpha_i$ of the composition are referred to as the parts
of the composition.  The size of the composition is the sum of the parts
and will be denoted $|\alpha|:=n$.  The length of the composition is the
number of parts and will be denoted $\ell(\alpha):=m$.
In this paper we study 
dual graded Hopf algebras whose bases at level $n$ are indexed by compositions of $n$.

Compositions of $n$ are in bijection with subsets of $\{1, 2, \dots, n-1\}$. We will follow the convention of identifying $\alpha = [\alpha_1, \alpha_2, \dots, \alpha_m]$ with the subset $D(\alpha) = 
\{\alpha_1, \alpha_1+\alpha_2, \alpha_1+\alpha_2 + \alpha_3, \dots, \alpha_1+\alpha_2+\dots + \alpha_{m-1} \}$. 

If $\alpha$ and
$\beta$ are both compositions of $n$,  say that $\alpha \leq  \beta$ in refinement order if $D(\beta) \subseteq D(\alpha)$. For instance, $[1,1,2,1,3,2,1,4,2] \leq [4,4,2,7]$, since $D([1,1,2,1,3,2,1,4,2]) = \{1,2,4,5,8,10,11,15\}$ and $D([4,4,2,7]) = \{4,8,10\}$.

We  introduce  a  new  notion  which  will  arise  in  our  
Pieri  rule  (Theorem  \ref{thm:Pieri});  we  say  that  $\alpha  \subset_{i}  \beta$  if:
\begin{enumerate}
\item $|\beta| = |\alpha| + i$,
\item $\alpha_j \leq \beta_j$ for all $1 \leq j \leq \ell(\alpha)$,
\item $\ell(\beta) \leq \ell(\alpha) + 1.$
\end{enumerate}
For a composition $\alpha = [\alpha_1, \alpha_2, \dots, \alpha_\ell]$ and a positive integer $m$, we let $[m, \alpha]$ stand for the composition $[m, \alpha_1, \alpha_2, \dots, \alpha_\ell]$.

In this presentation, compositions will be represented as diagrams of left aligned rows of cells.
The combinatorics of the elements that we introduce will lead us to represent our diagrams in this
way rather than as a ribbon (as is the usual method for representing compositions when working with
the ribbon Schur and fundamental bases).  We will also use the matrix convention (`English' notation)
that the first row of the diagram is at the top and the last row is at the bottom.  For example, the composition
$[4,1,3,1,6,2]$ is represented as
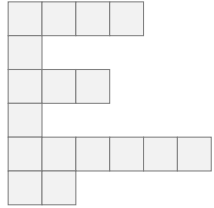
\[ \tikztableausmall{{X, X, X, X},{X}, {X, X, X}, {X}, {X,X,X,X,X,X}, {X, X}}~.
\]

Note that in examples and in a few formulas, compositions used as subscripts indexing elements of an algebra
will often be written without enclosing brackets.

\subsection{Schur functions and creation operators}

We use the standard notation for the common bases of $\sym$:
$h_\lambda$ for complete homogeneous;
$e_\lambda$ for elementary;
$m_\lambda$ for monomial;
$p_\lambda$ for power sums;
$s_\lambda$ for Schur.
For simplicity, we let $h_i$, $e_i$, $m_i$, $p_i$ and $s_i$ denote the
corresponding generators indexed by the partition $[i]$.

We next define a Schur function indexed by an arbitrary tuple of integers.
The family of symmetric functions indexed by partitions $\lambda$ are the usual
Schur basis of the symmetric functions.

\begin{Definition} \label{def:JTformula}
For an arbitrary integer tuple 
$\alpha = [\alpha_1, \alpha_2, \ldots, \alpha_\ell] \in \ZZ^\ell$, we define 
\[ s_\alpha := \det \begin{bmatrix} 
h_{\alpha_1}&h_{\alpha_1+1}&\cdots&h_{\alpha_1+\ell-1}\\
h_{\alpha_2-1}&h_{\alpha_2}&\cdots&h_{\alpha_2+\ell-2}\\
\vdots&\vdots&\ddots&\vdots\\
h_{\alpha_\ell-\ell+1}& h_{\alpha_\ell-\ell+2}&\cdots&h_{\alpha_\ell}\\
\end{bmatrix} = 
\det [h_{\alpha_i + j - i}]_{1 \leq i,j \leq \ell}\]
where we use the convention that $h_0 = 1$ and $h_{-m} = 0$ for $m>0$.
\end{Definition}

With this definition, we notice that switching two adjacent rows of the defining matrix
has the effect of changing the sign of the determinant.  
Switching rows of the matrix implies that we have the following equality:
$$s_{\alpha_1, \alpha_2, \ldots, \alpha_r, \alpha_{r+1}, \ldots, \alpha_\ell} =
-s_{\alpha_1, \alpha_2, \ldots, \alpha_{r+1}-1,\alpha_r+1, \ldots, \alpha_\ell} ~.
$$

Two rows of the matrix are equal if $\alpha_i - i = \alpha_j-j$.  This implies
part of the following well--known result.

\begin{Proposition} \label{prop:schurcomposition}
If $\alpha$ is a composition of $n$ with length equal to $k$, then
$s_\alpha = 0$ if and only if there exists $i,j \in \{ 1,2, \ldots,k\}$ with $i \neq j$ 
such that $\alpha_i -i = \alpha_j-j$. If this is not the case, then there is a unique permutation $\sigma$ such that $[\alpha_{\sigma_1} +1-\sigma_1, \alpha_{\sigma_2}+2-\sigma_2, \ldots, \alpha_{\sigma_k}+k-\sigma_k]$
is a partition.  In this case,
$$s_\alpha = (-1)^\sigma s_{\alpha_{\sigma_1} +1-\sigma_1, \alpha_{\sigma_2}+2-\sigma_2, \ldots, \alpha_{\sigma_k}+k-\sigma_k}~.$$
\end{Proposition}

$\sym$ is a self dual Hopf algebra. It has a pairing (the Hall scalar product)
defined by 
$$\langle h_\lambda, m_\mu \rangle =
\langle s_\lambda, s_\mu \rangle = \delta_{\lambda, \mu}~.$$
An element $f \in \sym$ gives rise to an operator $f^\perp: \sym \to \sym$ 
according to the relation:
\[ \langle fg,h \rangle = \langle g, f^\perp h \rangle \hspace{.1in} \textrm{ for all } g, h \in \sym.\]
Using this as a definition, the action of the operator $f^\perp$ can be 
calculated on another symmetric function by the formula
$$f^\perp(g) = \sum_{\lambda} \langle g, f a_\lambda \rangle b_\lambda$$
where $\{ a_\lambda \}_{\lambda}$ and $\{ b_\lambda \}_\lambda$ are any two
bases which are dual with respect to the pairing $\langle\cdot,\cdot\rangle$.

We define a ``creation'' operator $\B_m: \sym_n \to \sym_{m+n}$ by: \[\B_m := \sum_{i \geq 0} (-1)^i h_{m+i} e_i^\perp.\]

The following theorem, which states that creation operators construct Schur functions, will become one of the motivations 
for our new basis of $\Nsym$ (see Definition \ref{def:immaculate}).

\begin{Theorem} (Bernstein \cite[pg 69-70]{Zel}) \label{th:bern}
For all tuples $\alpha \in \ZZ^m$, 
\[ s_\alpha = \B_{\alpha_1} \B_{\alpha_2} \cdots \B_{\alpha_{m}}(1).\]
\end{Theorem}

Because of this result we shall refer to the $\B_m$ operators as either creation or Bernstein operators.

\subsection{Non-commutative symmetric functions}

The algebra of non-commutative symmetric functions $\Nsym$ is a non-commutative analogue of $\sym$ that arises by
considering an algebra with one non-commutative generator at each positive
degree.  In addition to the relationship with the symmetric functions,
this  algebra  has  links  to  Solomon's  descent  algebra  in  type  $A$  \cite{MR},
the  algebra  of  quasi-symmetric  functions  \cite{MR},  and representation theory
of  the  type  $A$  Hecke  algebra  at  $q=0$  \cite{KT}.  It is an example
of a combinatorial Hopf algebra \cite{ABS}.  While we will follow the foundational
results  and  definitions  from  references  such  as  \cite{GKLLRT,MR},  we  have  chosen
to use notation here which is suggestive of analogous results in $\sym$.

We define $\Nsym$ as the algebra with generators $\{\HH_1, \HH_2, \dots \}$ and 
no relations. Each generator $H_i$ is defined to be of degree $i$, 
giving $\Nsym$ the structure of a graded algebra. We let $\Nsym_n$ denote the 
graded component of $\Nsym$ of degree $n$. A basis for $\Nsym_n$ are the 
\textit{complete homogeneous functions} 
$\{\HH_\alpha := \HH_{\alpha_1} \HH_{\alpha_2} \cdots \HH_{\alpha_m}\}_{\alpha \vDash n}$ 
indexed by compositions of $n$.  To make this convention consistent, 
some formulas will use expressions that have $H$ indexed by tuples of integers
and we use the convention that $\HH_0=1$ and $\HH_{-r} = 0$ for $r>0$.

There exists a map (sometimes referred to as the forgetful map) which we shall
also denote $\chi: \Nsym \to \sym$ defined by sending the basis element  
$\HH_\alpha$ to the complete homogeneous symmetric function 
\begin{equation}\label{def:chi}
\chi(\HH_\alpha) := h_{\alpha_1} h_{\alpha_2} \cdots h_{\alpha_{\ell(\alpha)}} \in \sym
\end{equation}
and extended linearly to all of $\Nsym$. 
This map is a surjection onto $\sym$.

Similar to the study of $\sym$ and the ring of characters for the symmetric groups, the ring of non-commutative symmetric functions is isomorphic to the Grothendieck ring of finitely generated indecomposable projective representations of the $0$-Hecke algebra. We state this fact only as an analogy to $\sym$, we will not use it in this paper. We refer the reader to \cite{KT} for details.

The element of $\Nsym$ which corresponds to the indecomposable projective representation indexed by $\alpha$ is here denoted $R_\alpha$. The collection of $R_\alpha$ are a basis of $\Nsym$, usually called the \textit{ribbon basis} of $\Nsym$. They are defined through their expansion in the complete homogeneous basis: 
\begin{equation}\label{def:Ribbon} R_\alpha = \sum_{\beta \geq \alpha} (-1)^{\ell(\alpha)-\ell(\beta)} \HH_\beta
\hskip .2in \textrm{or equivalently} \hskip .2in \HH_\alpha = \sum_{\beta \geq \alpha} R_\beta. \end{equation}

The product expansion follows easily from the non-commutative product on the generators 
\[\HH_\alpha \HH_\beta = 
\HH_{\alpha_1, \ldots \alpha_{\ell(\alpha)},\beta_1, \ldots \beta_{\ell(\beta)}}~.\]
$\Nsym$ has a coalgebra structure, which is defined on the generators by
\[ \Delta( \HH_j ) = \sum_{i=0}^j \HH_i \otimes \HH_{j-i}~. \]
This determines the action of the coproduct on the basis $\HH_\alpha$ since the
coproduct is an algebra morphism with respect to the product. Explicitly we have:
\[\Delta( \HH_\alpha ) =
\Delta( \HH_{\alpha_1} ) \Delta( \HH_{\alpha_2} ) \cdots \Delta( \HH_{\alpha_{\ell(\alpha)}} )~.\]

\subsection{Quasi-symmetric functions}

The algebra of quasi-symmetric functions, $\Qsym$, was introduced in \cite{Ges} 
(see also subsequent
references such as \cite{GR, Sta84}) and this algebra has become a useful tool for algebraic
combinatorics since it is
an algebra which is dual to $\Nsym$ and which contains $\sym$ as a subalgebra. 

As  with  the  algebra  $\Nsym$,  the  graded  component  $\Qsym_n$  is  indexed  by  compositions  of  $n$. 
The algebra is most readily realized within the 
ring of power series of bounded degree 
$\mathbb{Q}[\![x_1, x_2, \dots]\!]$. The monomial 
quasi-symmetric function indexed by a composition $\alpha$ is defined as
\begin{equation}
    \label{monomial-qsym}
    M_\alpha = \sum_{i_1 < i_2 < \cdots < i_m} x_{i_1}^{\alpha_1} x_{i_2}^{\alpha_2} \cdots x_{i_m}^{\alpha_m}.
\end{equation}
The algebra of quasi-symmetric functions, $\Qsym$, can then be defined as the algebra with the 
monomial quasi-symmetric functions as a basis, whose multiplication is inherited as a subalgebra of $\mathbb{Q}[\![x_1, x_2, \dots]\!]$.
We define the coproduct on this basis as:
\[ \Delta(M_\alpha) = \sum_{S \subset \{1,2, \ldots, \ell(\alpha)\}} M_{\alpha_{S}} \otimes M_{\alpha_{S^c}},\]
where if $S=\{ i_1 < i_2 < \cdots < i_{|S|}\}$, then 
$\alpha_S = [\alpha_{i_1}, \alpha_{i_2}, \ldots, \alpha_{i_{|S|}}]$.

We view $\sym$ as a subalgebra of $\Qsym$. In fact, the usual monomial symmetric functions $m_\lambda \in \sym$
expand positively in the quasi-symmetric monomial functions :
\[ m_\lambda = \sum_{\sort(\alpha) = \lambda} M_\alpha,\]
where $\sort(\alpha)$ is the partition obtained by organizing the parts of $\alpha$ from the largest to the smallest.

Similar to $\Nsym$, the algebra $\Qsym$ is isomorphic to the Grothendieck ring of finite-dimensional representations of the $0$-Hecke algebra. The irreducible representations of the $0$-Hecke algebra form a basis for this ring, and under this isomorphism the irreducible representation indexed by $\alpha$ is identified with an element of $\Qsym$, the \textit{fundamental quasi-symmetric function}, denoted $F_\alpha$. The $F_\alpha$, for $\alpha \models n$, form a basis of $\Qsym_n$, and are defined by their expansion
in the monomial quasi-symmetric basis: 
\[F_\alpha = \sum_{\beta \leq \alpha} M_\beta.\]
\subsection{Identities relating non-commutative and quasi-symmetric functions} \label{sec:nsymqsymcalc}
The algebras $\Qsym$ and $\Nsym$ form graded dual Hopf algebras. The monomial basis of $\Qsym$ is dual in this context to the complete homogeneous basis of $\Nsym$, and the fundamental basis of $\Qsym$ is dual to the ribbon basis of $\Nsym$.
$\Nsym$ and $\Qsym$ have a pairing $\langle \cdot, \cdot \rangle: \Nsym \times \Qsym \to \mathbb{Q}$, defined under this duality as either $\langle \HH_\alpha, M_\beta \rangle = \delta_{\alpha, \beta}$, or $\langle R_\alpha, F_\beta \rangle = \delta_{\alpha, \beta}$.

We will generalize the operation which is dual to multiplication by a quasi-symmetric function
using this pairing.  For $F \in \Qsym$, let $F^\perp$ be the operator which acts on elements
$H \in \Nsym$ according to the relation $\langle H, F G \rangle = \langle F^\perp H, G \rangle$.  To expand 
$F^\perp(H)$, we take a basis $\{ A_\alpha \}_\alpha$ of $\Qsym$ and $\{ B_\alpha \}_\alpha$ a basis
of $\Nsym$ such that $\langle B_\alpha, A_\beta \rangle = \delta_{\alpha\beta}$, then

$$F^\perp(H) = \sum_{\alpha} \langle H, F A_\alpha \rangle B_\alpha~.$$

By the duality of the product and the coproduct structure of $\Nsym$ and $\Qsym$, 
we have for $F,G \in \Qsym$ and $H,K \in \Nsym$, that the
pairing satisfies 
$\langle H K, G \rangle = \langle H \otimes K, \Delta(G) \rangle$ and
$\langle H, F G \rangle = \langle \Delta(H), F \otimes G \rangle$.
As a consequence of the first of these two identities, we have the following Lemma.

\begin{Lemma} \label{lemma:perpcomult}
If $G \in \Qsym$ and $\Delta(G) = \sum_{i} G^{(i)} \otimes G_{(i)}$, then
for $H,K \in \Nsym$,
\[
G^\perp( H K ) = \sum_i G^{(i)\perp}(H) G_{(i)}^\perp(K)~. \]
\end{Lemma}
\begin{proof}
Because of the duality of the Hopf algebra structure between $\Nsym$ and $\Qsym$,
\begin{align*}
G^\perp( H K ) &= \sum_{\alpha} \langle G^\perp( H K ), A_\alpha \rangle B_\alpha
=\sum_{\alpha} \langle H K, G A_\alpha \rangle B_\alpha\\
&=\sum_{\alpha} \langle H \otimes K, \Delta(G A_\alpha) \rangle B_\alpha
=\sum_{\alpha} \langle H \otimes K, \Delta(G) \Delta(A_\alpha) \rangle B_\alpha\\
&=\sum_{\alpha} \sum_i \langle H \otimes K, (G^{(i)} \otimes G_{(i)}) \Delta(A_\alpha) \rangle B_\alpha\\
&=\sum_{\alpha} \sum_i \langle G^{(i)\perp}(H) G_{(i)}^\perp(K), A_\alpha \rangle B_\alpha~. \qedhere
\end{align*}
\end{proof}

Another way we can compute the action of $M_\alpha^\perp$ is by using the following Lemma.
\begin{Lemma}\label{lemma:Malphaperp}
For $G \in \Nsym$, if the coproduct on $G$ has the expansion
$\Delta(G) = \sum_{\gamma} \HH_\gamma \otimes G^{(\gamma)}$, then
$M_\alpha^\perp(G) = G^{(\alpha)}$~.
\end{Lemma}
\begin{proof} Let $G$ be an element of $\Nsym$ such that
$\Delta(G)$ has the expansion $\sum_{\gamma} \HH_\gamma \otimes G^{(\gamma)}$
with the $\HH$-basis in the left tensor.  A direct computation shows
\begin{align*}
M_\alpha^\perp(G) &= \sum_{\beta} \langle G, M_\alpha M_\beta \rangle \HH_\beta
=  \sum_{\beta} \langle \Delta(G), M_\alpha \otimes M_\beta \rangle \HH_\beta\\
&=  \sum_{\beta} \sum_\gamma \langle \HH_\gamma \otimes G^{(\gamma)}, M_\alpha \otimes M_\beta \rangle \HH_\beta=  \sum_{\beta} \langle  G^{(\alpha)}, M_\beta \rangle \HH_\beta = G^{(\alpha)}~.\qedhere
\end{align*}
\end{proof}

To develop some of the formulas for the immaculate basis we will need some
algebraic identities on $\Nsym$ and $\Qsym$.  These are standard results which
are analogous to the corresponding formulas in $\sym$, but require some development
of the algebra to verify their correctness.

As a consequence of Lemma \ref{lemma:perpcomult} we have the following relations.

\begin{Lemma}\label{lemma:Frule} For $i,j>0$ and for $f \in \Nsym$,
\begin{equation}\label{eq:eperpexpr}
 F_{1^i}^\perp( f \HH_j ) = F_{1^i}^\perp(f) \HH_{j} + F_{1^{i-1}}^\perp(f) \HH_{j-1}
\end{equation}
\begin{equation}\label{eq:hperpexpr}
F_{i}^\perp( f \HH_j ) = \sum_{k=0}^{min(i,j)} F_{i-k}^\perp(f) \HH_{j-k}~.
\end{equation}
\end{Lemma}

\begin{proof}
Since $F_i$ and $F_{1^i}$
are, respectively, $h_i$ and $e_i$ in $\sym$,
we know the coproduct rule $\Delta( F_{1^i} ) = \sum_{k=0}^{i} F_{1^{i-k}} \otimes F_{1^k}$,
and $\Delta( F_{i} ) = \sum_{k=0}^{i} F_{{i-k}} \otimes F_{k}$. 
From Lemma \ref{lemma:Malphaperp} and the fact that 
$F_{r}^\perp = \sum_{\alpha \models r} M_\alpha^\perp$, we calculate that 
$F_{r}^\perp( \HH_{j} ) = \sum_{\alpha\models r} M_\alpha^\perp(\HH_j) = \HH_{j-r}$ for $1\leq r\leq j$.
Another application of Lemma \ref{lemma:Malphaperp} with $F_{1^r}^\perp = M_{1^r}^\perp$
 shows that for $s>1$,
$F_{1^s}^\perp( \HH_{j} ) = M_{1^s}^\perp( \HH_j ) = 0$.  Equations
\eqref{eq:eperpexpr} and \eqref{eq:hperpexpr} are a consequence of Lemma \ref{lemma:perpcomult}.
\end{proof}

By applying the formulas in the previous lemma, the following expansions
may be shown by induction on the length of the composition.

\begin{Corollary}\label{lemma:ehperp}
For all $i \ge 0$ and all compositions $\alpha \models n$ such that $\ell(\alpha)=m$, 
\begin{equation}\label{eq:eperponHbasis}
F_{1^i}^\perp(\HH_\alpha) = \sum_{\substack{\beta \in \NN^m\\|\beta| = |\alpha| - i\\\alpha_j - 1 \leq \beta_j \leq \alpha_j}} \HH_\beta,
\end{equation}
\begin{equation}\label{eq:hperponHbasis}
F_{i}^\perp(\HH_\alpha) = \sum_{\substack{\gamma \in \NN^m\\|\gamma| = |\alpha| - i\\0 \leq \gamma_j \leq \alpha_j}} \HH_\gamma~.
\end{equation}
In both sums we use the convention that parts of size $0$ are deleted from the tuple
when it is re-expressed in the $\HH$-basis because $\HH_0 = 1$.
\end{Corollary}

\begin{Example}
We compute $F_{11}^\perp \HH_{2112} = \HH_{22} + 2\HH_{211} + 2\HH_{112} + \HH_{1111}$. This comes from removing two boxes from $[2,1,1,2]$ in the following ways and not considering parts of size $0$:

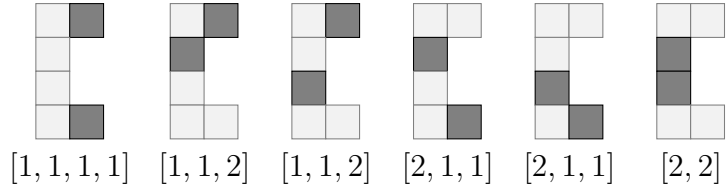
\begin{center}
$
\begin{array}{cccccc} 
\tikztableausmall{{X,\boldentry X},{X}, {X}, {X,\boldentry X}}  &
\tikztableausmall{{X,\boldentry X},{\boldentry X}, {X}, {X,X}} &
\tikztableausmall{{X,\boldentry X},{X}, {\boldentry X}, {X,X}} &
\tikztableausmall{{X,X},{\boldentry X}, {X}, {X,\boldentry X}}  &
\tikztableausmall{{X,X},{X}, {\boldentry X}, {X,\boldentry X}} &
\tikztableausmall{{X,X},{\boldentry X}, {\boldentry X}, {X,X}}
\\
[12pt] [1,1,1,1] &[1,1,2] & [1,1,2]  & [2,1,1] & [2,1,1] &  [2,2]
\end{array}
$
\end{center}

\end{Example}

\begin{Example}
We compute $F_2^\perp \HH_{2112} = 	
\HH_{1111} + 3\HH_{112} + 3\HH_{211} + \HH_{22}
$. This comes from removing two boxes from $[2,1,1,2]$ in the following ways:

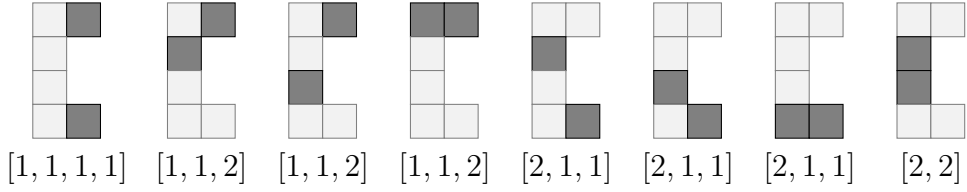
\begin{center}
$
\begin{array}{cccccccc} 
\tikztableausmall{{X,\boldentry X},{X}, {X}, {X,\boldentry X}}  &
\tikztableausmall{{X,\boldentry X},{\boldentry X}, {X}, {X,X}} &
\tikztableausmall{{X,\boldentry X},{X}, {\boldentry X}, {X,X}} &
\tikztableausmall{{\boldentry X,\boldentry X},{X}, {X}, {X, X}}  &
\tikztableausmall{{X,X},{\boldentry X}, {X}, {X,\boldentry X}} &
\tikztableausmall{{X,X},{X}, {\boldentry X}, {X,\boldentry X}} &
\tikztableausmall{{X,X},{X}, {X}, {\boldentry X,\boldentry X}} &
\tikztableausmall{{X,X},{\boldentry X}, {\boldentry X}, {X,X}} 
\\
[12pt] [1,1,1,1] & [1,1,2] & [1,1,2] & [1,1,2] & [2,1,1] & [2,1,1] & [2,1,1] & [2,2]
\end{array}
$
\end{center}

\end{Example}

We call a linear ordering of variables $(y_1,y_2,...)$ an {\sl alphabet}. Remark that an alphabet could be finite or countable.

For an alphabet $Y$, we define
    $$\OM_Y = \sum_{\alpha} M_{\alpha}[Y] \HH_\alpha,$$
where $M_\alpha[Y]$ is the monomial quasi-symmetric function indexed by $\alpha$
expanded  over  an    alphabet  $Y$ and the monomial quasi-symmetric functions are allowed
to commute with the non-commutative symmetric functions.    In  particular  when  $Y$  is  an  alphabet 
with a single variable,
\begin{gather*}
    \OM_z = \sum_{d \geq 0} z^d \HH_d.
\end{gather*}

If $Y = (y_1, y_2, \dots)$ is an alphabet and $z$ is another variable not in $Y$,
we let $z,Y$ denote the alphabet $(z, y_1, y_2, \dots)$.
We notice that $\OM_z \OM_Y = \OM_{z,Y}$ by calculating

\begin{align*}
    \OM_{z,Y}
    &= \sum_{\alpha} M_{\alpha}[z,Y] \HH_\alpha \\
    &= \sum_{\alpha}
        \left(M_{\alpha}[Y] + z^{\alpha_1} M_{\alpha_2, \dots, \alpha_m}[Y]\right)
        \HH_{\alpha} \\
    &= \sum_{\alpha} M_{\alpha}[Y] \HH_\alpha
    + \sum_{\alpha} z^{\alpha_1} \HH_{\alpha_1} M_{\alpha_2, \dots, \alpha_m}[Y] \HH_{\alpha_2, \dots, \alpha_m} \\
    &= \sum_\alpha M_\alpha[Y] \HH_\alpha 
     + \sum_{d > 0} z^{d} \HH_{d} \sum_{\beta} M_{\beta} [Y] \HH_{\beta} \\
    &= \left(\sum_{d \geq 0} z^d \HH_d\right) \left(\sum_\gamma M_\gamma[Y] \HH_\gamma\right) \\
    &= \OM_z \OM_Y.
\end{align*}
We can then determine by induction that $\OM_Z \OM_Y = \OM_{Z,Y}$ for a finite alphabet $Z$.

For $F \in \Qsym$, $F^\perp$ acts on the non-commutative symmetric functions and
does not affect the monomial quasi-symmetric function coefficients:
\begin{align*}
F^\perp \OM_Y &= \sum_{\alpha} M_\alpha[Y] F^\perp (\HH_\alpha) \\
&= \sum_{\alpha}M_\alpha[Y] \sum_\beta \langle \HH_\alpha, F M_\beta \rangle \HH_\beta\\
&= \sum_\beta \sum_{\alpha} \langle \HH_\alpha, F M_\beta \rangle M_\alpha[Y] \HH_\beta\\
&= \sum_\beta F[Y] M_\beta[Y] \HH_\beta  = F[Y] \OM_Y~.
\end{align*}

We also define the two operators with the parameter $z$ as
\begin{equation*}
    \EEE^\perp_{z} = \sum_{i \geq 0} z^i F_{1^i}^\perp,\hbox{ and }
    \HHH^\perp_{z} = \sum_{i \geq 0} z^i F_{i}^\perp,
\end{equation*} 
then $\EEE^\perp_{z}\OM_{Y} = \sum_{i \geq 0} z^i F_{1^i}[Y] \OM_Y = \OM_Y \prod_{y \in Y} (1+zy)$ and
$\HHH^\perp_{z}\OM_{Y} = \sum_{i \geq 0} z^i F_{i}[Y] \OM_Y = \OM_Y /\prod_{y \in Y}(1-zy)~.$

\section{A new basis for $\Nsym$}

We are now ready to introduce our new basis of $\Nsym$.
These functions were discovered while playing with a non-commutative analogue of the 
Jacobi-Trudi identity (Theorem \ref{thm:JT}). They may also be defined as the 
unique functions in $\Nsym$ which satisfy a right-Pieri rule (Theorem \ref{thm:Pieri} 
and Proposition \ref{prop:ePieri}). In order to streamline our proofs and extend our 
definitions to a Hall-Littlewood analogue, we start by building our new basis 
using a non-commutative version of the Bernstein operators (Theorem \ref{th:bern}).

\subsection{Non-commutative immaculate functions}

We continue with the notation of the previous section:
$\HH_i$ is the complete homogeneous non-commutative symmetric function;
$F_\alpha$ is the fundamental quasi-symmetric function indexed by the composition $\alpha$;
and
$F_\alpha^\perp$ is the linear transformation of $\Nsym$ that is adjoint
to multiplication by $F_\alpha$ in $\Qsym$.

\begin{Definition}
We define the non-commutative Bernstein operators $\BB_m$ as:
\[ \BB_m = \sum_{i \ge 0} (-1)^i \HH_{m+i} F_{1^i}^\perp~.\]
\end{Definition}

Using the non-commutative Bernstein operators, we can inductively build functions 
using creation operators similar to Bernstein's formula (Theorem \ref{th:bern}) for the Schur functions.

\begin{Definition}\label{def:immaculate}
For any $\alpha = [\alpha_1, \alpha_2, \cdots, \alpha_m] \in \ZZ^m$, 
the \emph{immaculate function} $\fS_\alpha \in \Nsym$ 
is defined as the composition of the operators applied to $1$ in the expression
$$\fS_\alpha := \BB_{\alpha_1} \BB_{\alpha_2} \cdots \BB_{\alpha_m} (1)~.$$
\end{Definition}

Calculations in the next subsection will show that the elements
$\{ \fS_\alpha \}_{\alpha \models n}$ form a basis for $\Nsym_n$.

\begin{Example}
If $\alpha = (a)$ has only one part, then $\fS_a$ is just the complete homogeneous generator $\HH_a$. If $\alpha = [a, b]$ consists of two parts, then $\fS_{ab} = \BB_a (\HH_b) = \HH_a \HH_b - \HH_{a+1} \HH_{b-1}$.
\end{Example}

\subsection{The right-Pieri rule for immaculate functions}

\begin{Lemma}\label{lemma:Bexpansion} For $s \geq 0$ and $m \in \ZZ$, and for $f$ an element of $\Nsym$,
\[ \BB_m(f)\HH_s = \BB_{m+1}(f) \HH_{s-1} + \BB_m(f\HH_s).\]
\end{Lemma}

\begin{proof}
By definition:
$$  \BB_m (f \HH_s) =  \sum_{i\ge 0} (-1)^i \HH_{m+i} F_{1^i}^\perp (f \HH_s).$$
Using Lemma \ref{lemma:Frule}, one obtains:
     $$  \BB_m (f \HH_s) =  \sum_{i\ge 0} (-1)^i \HH_{m+i} (F_{1^{i-1}}^\perp (f) \HH_{s-1} + F_{1^i}^\perp (f) \HH_s ),$$
which by associativity and reindexing gives:
     $$\BB_m (f \HH_s) = - \BB_{m+1}(f) \HH_{s-1} + \BB_m(f) \HH_s.\qedhere$$
\end{proof}
\begin{Theorem}\label{thm:Pieri} For a composition $\alpha$, 
the $\fS_\alpha$ satisfy a multiplicity free right-Pieri rule for multiplication by $\HH_s$:
\[\fS_\alpha  \HH_s  =  \sum_{  \alpha  \subset_{s}  \beta}  \fS_\beta,\] 
where the notation $\subset_{s}$ is introduced in Section \ref{sec:compositions}.
\end{Theorem}

\begin{proof}
Let $m = \alpha_1$ and let $\overline{\alpha} = [\alpha_2, \alpha_3, \dots, \alpha_k]$ denote the composition with first part removed. The proof will be by induction on $s + \ell(\alpha)$, the base case being trivial. By definition, $\BB_m \fS_{\overline{\alpha}} = \fS_\alpha$, so

\begin{align*}
 \fS_\alpha \HH_s &= \BB_m (\fS_{\overline{\alpha}}) \HH_s \\
&= \BB_{m+1}(\fS_{\overline{\alpha}})\HH_{s-1} + \BB_m(\fS_{\overline{\alpha}} \HH_s) & \textrm{ by Lemma \ref{lemma:Bexpansion} },\\
&= \fS_{[m+1, \overline{\alpha}]} \HH_{s-1} + \BB_m(\sum_{\overline{\alpha} \subset_s \eta} \fS_{\eta})& \textrm{ by the Pieri rule on $\overline{\alpha}$ and $s$},\\
&= \sum_{[m+1, \overline{\alpha}] \subset_{s-1} \gamma} \fS_\gamma + \sum_{\overline{\alpha} \subset_s \eta} \fS_{[m, \eta]} & \textrm{ by the Pieri rule on $[m+1, \overline{\alpha}]$ and $s-1$}.
\end{align*}
The first sum counts all $\gamma$ which arise from adding boxes to $\alpha$, adding at least one to the first part of the composition, which are bounded in length by $\ell(\alpha) + 1$. The second sum counts all $\eta$ which arise from adding boxes to $\alpha$, without adding to the first part, which are bounded in length by $\ell(\alpha) + 1$. The statement now follows from combining the two sums.
\end{proof}

\begin{Remark}
Products of the form $\HH_m \fS_\alpha$ do not have as nice an expression as $\fS_\alpha \HH_m$
since they
generally have negative signs in their expansion and there is no obvious containment of
resulting compositions.  For example,
\[ \HH_1 \fS_{13} = \fS_{113} - \fS_{221} - \fS_{32}~. \]
The reason for this is that left multiplication by $\HH_m$
can be re-expressed as
\[ \HH_m = \sum_{i \geq 0} \BB_{m+i} F_{i}^\perp~. \]
We will develop the algebra required to understand where the negative signs potentially arise,
but we will not give a satisfactory left Pieri rule.
We conjecture that the left Pieri rule is multiplicity free, up to sign.
\begin{Conjecture}
\[\HH_m \fS_\alpha = \sum_\beta (-1)^{sign(\alpha, \beta)} \fS_\beta,\] where the sum is over some collection of compositions $\beta$ of size $|\alpha| +m$ and $sign$ is some statistic which depends on $\alpha$ and $\beta$.
\end{Conjecture}
\end{Remark}

\begin{Example}\label{pigeon} The expansion of $\fS_{23}$ multiplied on the right by $\HH_3$ is done below, along with corresponding pictures.
\[
\begin{array}{ccccccccccc}
\tikztableausmall{{X,X},{X,X,X}}& &\tikztableausmall{{\boldentry X,\boldentry X,\boldentry X}}  =   &
\tikztableausmall{{X, X},{X,X, X},{\boldentry X, \boldentry X, \boldentry  X}} &
 \tikztableausmall{{X, X},{X,X, X, \boldentry X},{\boldentry X,\boldentry X}} &
 \tikztableausmall{{X,X},{X,X,X,\boldentry X,\boldentry X},{\boldentry X}}\\
 \fS_{23} &* &\HH_3 = & \fS_{233} &+\, \fS_{242} &+\,\fS_{251}
\end{array}
\]
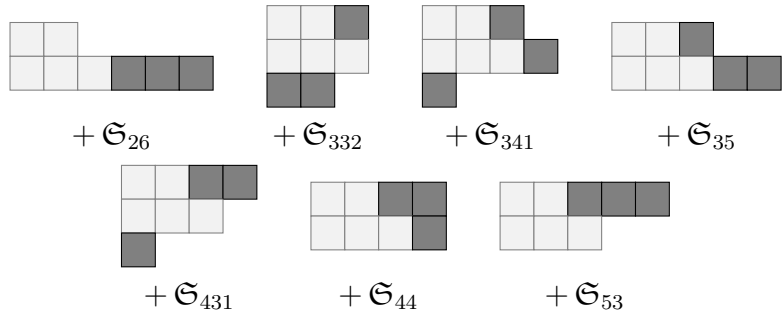
\[
\begin{array}{ccccc}
& \tikztableausmall{{X,X},{X, X, X, \boldentry X, \boldentry X, \boldentry X}} &
 \tikztableausmall{{X,X,\boldentry X},{X,X,X}, {\boldentry X, \boldentry X}} &
   \tikztableausmall{{X,X,\boldentry X},{X, X, X, \boldentry X}, {\boldentry X}} &
 \tikztableausmall{{X,X,\boldentry X},{X, X, X, \boldentry X, \boldentry X}}\\
&+\,\fS_{26} &+\, \fS_{332}& +\, \fS_{341}& +\, \fS_{35}
\end{array}
\]
\[
\begin{array}{ccccc}
& \tikztableausmall{{X,X,\boldentry X,\boldentry X},{X, X, X}, {\boldentry X}} 
& \tikztableausmall{{X,X,\boldentry X,\boldentry X},{X, X, X, \boldentry X}} 
& \tikztableausmall{{X,X, \boldentry X, \boldentry X, \boldentry X},{X,X,X}}
\\
&+\,\fS_{431} &+ \,\fS_{44} &+\, \fS_{53}
\end{array}
\]
\end{Example}

\subsection{Relationship with the classical bases of $\Nsym$}

We will now develop some relations between the classical bases of $\Nsym$ and the immaculate basis. In particular, the first result will establish the fact that the immaculate functions indexed by compositions do in fact form a graded basis of $\Nsym$. First, we need the notion of an immaculate tableau.

\subsubsection{Immaculate tableaux}

\begin{Definition}
Let $\alpha$ and $\beta$ be compositions. An \emph{immaculate tableau} of shape
$\alpha$ and content $\beta$ is a labelling of the boxes of the diagram of
$\alpha$ by positive integers in such a way that:
\begin{enumerate}
\item the number of boxes labelled by $i$ is $\beta_i$;
\item the sequence of entries in each row, from left to right, is weakly increasing;
\item the sequence of entries in the \emph{first} column, from top to bottom,
    is increasing.
\end{enumerate}

An immaculate tableau is said to be \emph{standard} if it has content
$1^{|\alpha|}$.

Let $K_{\alpha, \beta}$ denote the number of immaculate tableaux of shape
$\alpha$ and content $\beta$.
\end{Definition}

We re-iterate that, aside from the elements in the first column, there is no relation on the elements in the
columns of an immaculate tableau.

\begin{Example}\label{ex:immaculatetableau}
There are five immaculate tableaux of shape $[4,2,3]$ and content $[3,1,2,3]$: 
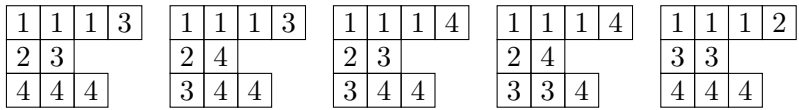
\[ \tikztableausmall{{1,1, 1, 3},{2, 3}, {4,4,4}} 
\tikztableausmall{{1,1, 1, 3},{2, 4}, {3,4,4}} 
\tikztableausmall{{1,1, 1, 4},{2,3}, {3,4,4}} 
\tikztableausmall{{1,1, 1, 4},{2, 4}, {3,3,4}} 
\tikztableausmall{{1,1, 1, 2},{3, 3}, {4,4,4}} 
\]
\end{Example}

Standard immaculate tableaux of size
$n$ can be identified with set partitions of $\{1, 2, \dots, n\}$ by ordering
the parts in the partition by minimal elements.  
This remark was pointed out to us in a discussion with Martha Yip.
It allows us to state a surprising enumeration formula for standard immaculate
tableaux which is analogous to the hook length formula for standard tableaux.
In order to state and prove this formula, we need to define
the standardization of immaculate tableaux and hooks of cells.

\begin{Definition} \label{def:standardization}
Given an immaculate tableau $T$ of shape $\alpha$ and content $\beta$, we form a standard immaculate tableau $std(T) = S$ of shape $\alpha$ and content $(1^n)$ as follows. We order the entries of the tableau $T$, reading first all entries valued $1$, then $2$, etc. Among all entries with the same value, we first read all entries in the lowest row, starting at the leftmost position and read first to the right and then up rows. The order of the entries forms a standard immaculate tableau which we call the standardization of $T$.
\end{Definition}

\begin{Example}
\label{example:standardization}
The following tableau is of  shape $[6, 5, 7]$ and content $[2, 3, 5, 1, 4, 3]$.
$$ T= \tikztableausmall{{1,1,2,2, 3,4},{2, 3, 3, 3, 3}, {5,5,5,5,6, 6, 6}} $$
The standardization of $T$ is:
$$ S= \tikztableausmall{{1,2,4,5, 10,11},{3, 6, 7, 8, 9}, {12,13,14,15,16, 17, 18}} $$
\end{Example}

Let $c = (i,j)$ be a cell in row $i$ and column $j$ of the
diagram for a composition $\alpha$ (that is, $1 \leq i \leq \ell(\alpha)$ and $1 \leq j \leq \alpha_i$).
If $c = (i,1)$, define the {\it hook} of $c$ in $\alpha$ to be $h_\alpha(c) = \alpha_i + \alpha_{i+1} + \cdots + \alpha_{\ell(\alpha)}$
(the number of cells below and to the right in the diagram).
If $j>1$, then the hook of $c$ in $\alpha$ is $h_\alpha(c) = \alpha_i -j + 1$
(the number of cells weakly to the right in the same row).

\begin{Proposition}
If $\alpha \models n$, the number of standard immaculate tableaux of shape $\alpha$ is equal to
\begin{equation}\label{eq:hooklength}
K_{\alpha,1^{n}} = \frac{n!}{\prod_{c \in \alpha} h_\alpha(c)}
\end{equation}
where $c \in \alpha$ indicates $c = (i,j)$ with $1 \leq i \leq \ell(\alpha)$ and $1 \leq j \leq \alpha_i$.
\end{Proposition}

\begin{proof} Consider a standard immaculate tableau of shape $[k,\alpha] \models n+k$.  The first row contains $1$
and a subset $S$ of $k-1$ other integers from $\{2, 3, \ldots, n+k \}$.  
Moreover, if we standardize rows $2$ through $\ell(\alpha)+1$ of this tableau
then we have a standard immaculate tableau of shape $\alpha$.
This gives us a bijection between standard immaculate tableaux of shape $[k,\alpha]$ and the set of pairs $(S,T)$
where $S$ is a subset of $\{2, 3, \ldots, n+k \}$ of size $k-1$ and $T$ is a
standard immaculate tableau of shape $\alpha$.
This is a bijective proof of the recursion
\begin{equation}
K_{[k,\alpha],1^{n+k}} = \binomial{n+k-1}{k-1} K_{\alpha,1^{n}}~.
\end{equation}

Now the hook length formula follows by an induction argument on the length of the composition $\alpha$.
Assume that we know Equation
\eqref{eq:hooklength} holds for compositions $\alpha$ of length $\ell$, then the
hooks of the cells in the first row of $[k, \alpha]$ are (respectively) $n+k, k-1, k-2, \ldots, 2, 1.$
\begin{align*}
K_{[k,\alpha],1^{n+k}} &= \binomial{n+k-1}{k-1} K_{\alpha,1^{n}} = \binomial{n+k-1}{k-1}
\frac{n!}{\prod_{c \in \alpha} h_\alpha(c)}\\
&= \frac{(n+k-1)(n+k-2)\cdots(n+1)}{(k-1)!} \frac{n!}{\prod_{c \in \alpha} h_\alpha(c)}\\
&= \frac{(n+k)(n+k-1)(n+k-2)\cdots(n+1)}{(n+k)(k-1)!} \frac{n!}{\prod_{c \in \alpha} h_\alpha(c)}\\
&=  \frac{(n+k)!}{\prod_{c \in [k,\alpha]} h_\alpha(c)}~.
\end{align*}
This shows that (\ref{eq:hooklength}) holds for compositions of length $\ell+1$.  The base case for $\ell=0$
 holds trivially.  Therefore the hook length formula (\ref{eq:hooklength}) holds for all compositions.
\end{proof}

\begin{Example}
The hook length formula (\ref{eq:hooklength}) says that since the hooks of $[4,2,3]$ are given by
the entries in the diagram:
\[ \tikztableausmall{{9,3, 2, 1},{5, 1}, {3,2,1}} \]
The number of standard immaculate tableaux of shape $[4,2,3]$ is equal to
$$\frac{9!}{9 \cdot 3 \cdot 2 \cdot 1 \cdot 5 \cdot 1 \cdot 3 \cdot 2 \cdot 1} = 224~.$$
\end{Example}

\begin{Proposition}\label{prop:kostkamatrix}
Recall that $K_{\alpha, \beta}$ denote the number of immaculate tableaux of shape
$\alpha$ and content $\beta$. Then
\begin{enumerate}
\item $K_{\alpha, \alpha} = 1,$
\item $K_{\alpha, \beta} = 0$ unless $\alpha \geq_\ell \beta$,
\end{enumerate}
where $\geq_\ell$ represents the lexicographic order on compositions.
\end{Proposition}
\begin{proof}
The first fact follows from definition: there is only one tableau with shape
and content $\alpha$, it consists of $\alpha_i$ many $i$'s in row $i$.
For the second statement, we argue by contradiction. Suppose there is such
a tableau and let $i$ be the first integer such that $\alpha_i < \beta_i$ (so that $\alpha_j
= \beta_j$ for $j < i$). As above, one must place $\beta_j$ many $j$'s in row
$j$, for $j<i$, filling all positions in the rows above row $i$. Then one must
place $\beta_i$ many $i$'s into row $i$, which contains only $\alpha_i$ spaces,
a contradiction.
\end{proof}

\subsubsection{Expansion of the homogeneous basis}

\begin{Proposition}\label{cor:Hexpansion}
The complete homogeneous basis $\HH_\alpha$ has a positive, uni-triangular expansion in the immaculate functions indexed by compositions. Specifically,
\[ \HH_\beta = \sum_{\alpha\geq_\ell \beta} K_{\alpha, \beta} \fS_\alpha.\]
\end{Proposition}

\begin{proof}
Follows from repeated application of the Pieri rule (Theorem \ref{thm:Pieri}). To see this, start with $H_{\beta_1} = \fS_{\beta_1}$, which is clearly a sum over immaculate tableaux of shape and content $\beta_1$. Each time one multiplies by $H_i$, by the Pieri rule, one adds to the immaculate tableaux in the index of summation, a set of entries labelled $i$ which satisfy the conditions of the definition of an immaculate tableau. The reason that
we need only sum over $\alpha \geq_\ell \beta$ is explained in Proposition \ref{prop:kostkamatrix}.
\end{proof}

\begin{Example}
Continuing from Example \ref{ex:immaculatetableau}, we see that \[ \HH_{3123} = \cdots + 5 \fS_{423} + \cdots.\]
\end{Example}

\begin{Corollary}
The $\{\fS_\alpha : \alpha \vDash n\}$  form a basis of $\Nsym_n$.
\end{Corollary}

\begin{proof}
    These elements span $\Nsym_n$ by Proposition \ref{cor:Hexpansion}
    and we have the correct number of elements, so they form a basis.
\end{proof}

\begin{Remark}
We will now use the term ``immaculate basis'' to mean those immaculate functions which are indexed by compositions.
\end{Remark}

\subsubsection{Expansion of the ribbon basis}

Next we will expand the ribbon functions in the immaculate basis. To do this, we first need the notions of  standardization and descent.

\begin{Definition} \label{def:descentSIT}
We say that a standard immaculate tableau
$T$ has a descent in position $i$
if  $(i+1)$ is in a row strictly lower than $i$ in $T$.  The symbol $D(T)$ will
represent the set of descents in $T$ and $\alpha(D(T))$ will represent the
corresponding composition.
\end{Definition}

\begin{Example}
    The standard immaculate tableau $S$ of
    Example~\ref{example:standardization} has descents in positions
    $\{2, 5, 11\}$.
    The descent composition of $S$ is then $[2,3,6,7]$.
\end{Example}

\begin{Remark}
The content of an immaculate tableau $T$ must be a refinement of the descent composition $\alpha(D(std(T)))$. This follows from the definition of standardization; the only places descents can happen in $std(T)$ are the final occurrence of a given content in $T$.
\end{Remark}

We let $L_{\alpha, \beta}$ denote the number of standard immaculate tableaux of shape $\alpha$ and descent composition $\beta$.

\begin{Lemma}\label{lemma:bijection}
 $$ K_{\alpha,\gamma} = \sum_{\beta\ge\gamma}   L_{\alpha,\beta}  .$$
\end{Lemma}

\begin{proof} Standardization provides a bijection between the set of immaculate tableaux of shape $\alpha$ and content $\gamma$ and the set of standard immaculate tableaux of shape $\alpha$ and descent composition $\beta \geq \gamma$. The result follows from this bijection.
\end{proof}

\begin{Example}
The five immaculate tableaux of shape $[4,2,3]$ and content $[3,1,2,3]$
$$
\begin{array}{ccccc}
\tikztableausmall{{1,1,1, 3},{2, 3}, {4,4,4}} &
\tikztableausmall{{1,1,1, 3},{2, 4}, {3,4,4}} &
\tikztableausmall{{1,1,1, 4},{2, 3}, {3,4,4}} &
\tikztableausmall{{1,1,1, 4},{2, 4}, {3,3,4}} &
\tikztableausmall{{1,1,1, 2},{3, 3}, {4,4,4}} 
\end{array}
$$
are bijected, under standardization, with the five standard immaculate tableaux
of shape $[4,2,3]$ and descent composition $\beta \geq [3,1,2,3]$ below.
$$
\begin{array}{ccccc}
\tikztableausmall{{1,2, 3, 6},{4, 5}, {7,8,9}} &
\tikztableausmall{{1,2, 3, 6},{4, 9}, {5,7,8}} &
\tikztableausmall{{1,2 , 3, 9},{4,6}, {5,7,8}} &
\tikztableausmall{{1,2, 3, 9},{4, 8}, {5,6,7}} &
\tikztableausmall{{1,2, 3, 4},{5, 6}, {7,8,9}} \\
[12 pt][3,3,3] & [3,1,2,3] & [3,1,2,3] & [3,1,5] & [4,2,3]
\end{array}
$$
\end{Example}

\begin{Theorem}\label{thm:Rpositive}
The ribbon function $R_\beta$ has a positive expansion in the immaculate basis. Specifically 
\[ R_\beta = \sum_{\alpha\geq_\ell\beta} L_{\alpha, \beta} \fS_\alpha.\]
\end{Theorem}

\begin{proof}
By Proposition \ref{cor:Hexpansion}, we know that $\HH_\gamma = \sum_\alpha K_{\alpha, \gamma} \fS_\alpha$, and we can substitute Lemma \ref{lemma:bijection} and obtain:
\[\HH_\gamma = \sum_\alpha K_{\alpha, \gamma} \fS_\alpha = \sum_\alpha \sum_{\beta \geq \gamma} L_{\alpha, \beta} \fS_\alpha  = \sum_{\beta \geq \gamma} \left(\sum_{\alpha} L_{\alpha,\beta} \fS_\alpha\right),\] which is the defining relation for the ribbon basis (see Equation \ref{def:Ribbon}).
We know that $L_{\alpha, \beta} \leq K_{\alpha, \beta}$, hence we need only sum over $\alpha \geq_\ell \beta$.
\end{proof}

\begin{Example}
There are eight standard immaculate tableaux with descent composition $[2,2,2]$, giving the expansion of $R_{222}$ into the immaculate basis.
$$
\begin{array}{cccccc}
&
\tikztableausmall{{1,2},{3, 4}, {5,6}} &
\tikztableausmall{{1,2},{3, 4,6}, {5}} &
\tikztableausmall{{1,2 , 4},{3}, {5,6}} &
\tikztableausmall{{1,2, 4},{3,6}, {5}}
\tikztableausmall{{1,2, 6},{3,4}, {5}}
\\ R_{222} = &
\fS_{222} &+\, \fS_{231} & +\, \fS_{312} &+\,2\fS_{321}
\end{array}
$$
$$
\begin{array}{cccc}
&
\tikztableausmall{{1,2,4},{3,5,6}}&
\tikztableausmall{{1,2 ,4 ,6},{3}, {5}}&
\tikztableausmall{{1,2, 4, 6},{3,5}}\\
&+\,\fS_{33} &+\, \fS_{411} &+\, \fS_{42}
\end{array}
$$
\end{Example}
\subsection{Jacobi-Trudi rule for $\Nsym$}

Another compelling reason to study the immaculate functions is that they have an expansion in
the $\HH_\alpha$ basis that makes them a clear analogue of the Jacobi-Trudi rule
of Definition \ref{def:JTformula}.

\begin{Theorem}\label{thm:JT} For $\alpha  \in  \ZZ^m$:
\begin{equation}\label{eq:JTformula}
\fS_\alpha  = \sum_{\sigma \in S_m} (-1)^\sigma \HH_{\alpha_1+\sigma_1 -1, \alpha_2 + \sigma_2 -2, \dots, \alpha_m + \sigma_m - m},
\end{equation}
where we have used the convention that $\HH_0 = 1$ and $\HH_{-m} = 0$ for $m>0$.
\end{Theorem}

\begin{Remark}
This sum is a non-commutative analogue of the determinant of the following matrix:

\[
\begin{bmatrix} 
\HH_{\alpha_1}&\HH_{\alpha_1+1}&\cdots&\HH_{\alpha_1+\ell-1}\\
\HH_{\alpha_2-1}&\HH_{\alpha_2}&\cdots&\HH_{\alpha_2+\ell-2}\\
\vdots&\vdots&\ddots&\vdots\\
\HH_{\alpha_\ell-\ell+1}& \HH_{\alpha_\ell-\ell+2}&\cdots&\HH_{\alpha_\ell}\\
\end{bmatrix}.\]
The non-commutative analogue of the determinant corresponds to expanding the
determinant of this matrix about the first row and multiplying those elements
on the left.
\end{Remark}

\begin{Remark} One might ask why one would naturally
expand about the first row rather than, say, the first column or the last row.  What
we considered to be the natural analogue of expanding about the first column however
is not a basis since, for instance, the matrix corresponding to $\alpha = [1,2]$
would be $0$ under this analogue.
\end{Remark}

Before we begin with the proof we introduce some notation that will prove useful
in our development of this identity.
Let $z$ be a single variable and define the operator $\BB(z) = \sum_{r \in \ZZ}
z^r \BB_r$. We will show that the coefficient of 
$z_1^{\alpha_1} z_2^{\alpha_2} \cdots z_m^{\alpha_m}$
in $\BB(z_1) \BB(z_2) \dots \BB(z_m) (1)$ is the right hand side
of \eqref{eq:JTformula}.

Using the notation developed in Section \ref{sec:nsymqsymcalc},
\begin{align}
\BB(z)
    &= \sum_{m \in \ZZ} z^m \BB_m
    \notag \\
    &= \sum_{m \in \ZZ} z^m \sum_{i \geq 0} (-1)^i \HH_{m+i} F_{1^i}^\perp
    \notag \\
    &= \sum_{d \geq 0} \sum_{i \geq 0} z^{d-i} (-1)^i \HH_{d} F_{1^i}^\perp
    \notag \\
    &= \left(\sum_{d \geq 0} z^d \HH_{d}\right)
        \left(\sum_{i \geq 0} (-1/z)^i F_{1^i}^\perp\right)\nonumber\\
    &= \OM_z \EEE_{-1/z}^\perp~.
    \label{Bz-factorization}
\end{align}

Hence, for every indeterminate $z_r$ and every alphabet $Y$ of commuting indeterminates which commute with $z_r$ but are distinct from $z_r$, we have
\begin{align}
    \label{Bz-action}
    \BB(z_r) \OM_Y
    = \OM_{z_r} \EEE_{-1/z_r}^\perp \OM_Y
    = \OM_{z_r} \OM_Y \prod_{y \in Y} (1-y/z_r)
    = \OM_{z_r, Y} \prod_{y \in Y} (1-y/z_r)~.
\end{align}

Repeated application of \eqref{Bz-action} yields that for any alphabet $Y$ of commuting indeterminates which commute with $z_1, z_2,\cdots, z_r$ but are distinct from each of them, we have
\begin{align}
    \label{prodBz-identity}
    \BB(z_1) \BB(z_{2}) \cdots, \BB(z_{m}) \ \OM_Y
    =
    \OM_{z_1, z_2, \dots, z_m, Y}
    \prod_{r=1}^m \left(\prod_{y \in \{z_{r+1}, \dots z_{m}\} \cup Y} (1 - y/z_r)\right)~.
\end{align}

\begin{proof}[Proof of Theorem \ref{thm:JT}]
Taking $Y = \emptyset$ in \eqref{prodBz-identity}, we obtain the identity
\begin{align}
    \label{prodBz-specialization}
    \BB(z_1) \BB(z_{2}) \cdots \BB(z_{m}) \,1
    =
    \OM_{z_1, z_2, \dots, z_m}
    \prod_{1 \leq i < j \leq m} \left(1 - z_j / z_i\right).
\end{align}
Using the Vandermonde determinant identity,
$$\prod_{1 \leq i < j \leq m} (z_i - z_j) = \sum_{\sigma \in S_m} (-1)^\sigma
z_1^{m-\sigma_1} z_2^{m-\sigma_2} \cdots z_m^{m-\sigma_m}~,$$
we can rewrite \eqref{prodBz-specialization} as
\begin{align*}
\BB(z_1) \BB(z_2) \cdots \BB(z_m)\, 1
&= \OM_{z_1,z_2,\ldots,z_m} \prod_{1\leq i<j\leq m} (1-z_j/z_i)\\
&= \OM_{z_1,z_2,\ldots,z_m} z_1^{1-m} z_2^{2-m} \cdots z_m^{m-m} \prod_{1\leq i<j\leq m} (z_i-z_j)\\
&= \OM_{z_1,z_2,\ldots,z_m} \sum_{\sigma \in S_m} (-1)^\sigma
z_1^{1-\sigma_1} z_2^{2-\sigma_2} \cdots z_m^{m-\sigma_m}~.
\end{align*}
So the coefficient of $z_1^{\alpha_1} z_2^{\alpha_2} \cdots z_m^{\alpha_m}$
in $\BB(z_1) \cdots \BB(z_m) (1)$ is equal to
\begin{align*}
\fS_\alpha &= \BB_{\alpha_1} \BB_{\alpha_2} \cdots \BB_{\alpha_m} \, 1 
= \BB(z_1) \BB(z_2) \cdots \BB(z_m)\, 1 \coeff_{z_1^{\alpha_1} z_2^{\alpha_2} \cdots 
z_m^{\alpha_m}}\\
&=\OM_{z_1,z_2,\ldots,z_m} \sum_{\sigma \in S_m} (-1)^\sigma
z_1^{1-\sigma_1} z_2^{2-\sigma_2} \cdots z_m^{m-\sigma_m} \coeff_{z_1^{\alpha_1} z_2^{\alpha_2} \cdots 
z_m^{\alpha_m}}\\
&= \sum_{\sigma \in S_m} (-1)^\sigma \OM_{z_1,z_2,\ldots,z_m}
\coeff_{z_1^{\alpha_1+\sigma_1-1} z_2^{\alpha_2+\sigma_2-2} \cdots 
z_m^{\alpha_m+\sigma_m-m}}\\
&=\sum_{\sigma \in S_m} (-1)^\sigma
\HH_{\alpha_1+\sigma_1-1,\alpha_2+\sigma_2-2, \cdots, \alpha_m+\sigma_m-m}~.
\qedhere
\end{align*}
\end{proof}

Of course, the original reason for considering this definition is
the property that they are a lift of the symmetric function corresponding
to the Jacobi-Trudi matrix.

\begin{Corollary}\label{thm:projection}
    For any composition $\alpha$, we have $\chi( \fS_\alpha) = s_\alpha$.
\end{Corollary}

\begin{proof} 
This follows from Definition \ref{def:JTformula} of the Jacobi-Trudi rule 
and the fact that $\chi(\HH_i) = h_i$.
\end{proof}

\subsection{A Pieri rule for the elementary basis}

Theorem \ref{thm:JT} shows that in particular for the case
of $\alpha = [1^n]$ that $\fS_{1^n}$ is the usual analogue of the elementary
generators of $\sym$.  The elementary generators of $\Nsym$ are the elements
$E_i$ which satisfy that $E_0=1$, and for $n \ge 1$,
\[ E_n := \sum_{i=1}^n (-1)^{i-1} \HH_{i} E_{n-i}.\]
The antipode map is both an algebra antimorphism and a coalgebra morphism so
that $S( H_{n}  ) = (-1)^{n} E_{n}$.

\begin{Corollary}\label{cor:equaltoEn}
For $n\geq0$, 
\[ \fS_{1^n} = \sum_{\alpha \models n} (-1)^{n-\ell(\alpha)} H_{\alpha} \]
and as a consequence, $F_{1^r}^\perp( \fS_{1^n}) = \fS_{1^{n-r}}$ and
for $s>1$, $F_{s}^\perp( \fS_{1^n} ) = 0$.
\end{Corollary}

\begin{proof}
The expansion of $\fS_{1^n}$ in terms of the complete homogeneous basis
is a direct consequence of Theorem \ref{thm:JT}.  
For $n>1$, expand $\fS_{1^n}$ in the leftmost occurrence of $\HH$
and group the rest of the terms together as
\[
 \fS_{1^n}
 = \sum_{i=1}^{n} \sum_{\beta \models n-i} (-1)^{n-\ell(\beta)-1} \HH_{i} \HH_{\beta}
 = \sum_{i=1}^{n} (-1)^{i-1} \HH_{i} \fS_{1^{n-i}}~.
\]
This implies that $\fS_{1^n} = E_n = (-1)^n S( \HH_n )$ is the usual analogue of the
elementary generators of $\Nsym$.  Because the antipode is a coalgebra morphism,
\[ \Delta(\fS_{1^n}) = (-1)^n \Delta( S( \HH_n )) = (-1)^n 
\sum_{i=0}^n S(H_{n-i}) \otimes S(H_{i}) = \sum_{i=0}^n \fS_{1^{n-i}} \otimes \fS_{1^i}~. \]
By Lemma \ref{lemma:Malphaperp} we know that $F_{1^r}^\perp( \fS_{1^{n}} ) =
M_{1^r}^\perp( \fS_{1^n} ) = \fS_{1^{n-r}}$.  Also by Lemma \ref{lemma:Malphaperp},
for $s>1$,
\[F_{s}^\perp( \fS_{1^n}) = \sum_{\alpha \models s} M_\alpha^\perp( \fS_{1^n} )
= \sum_{\alpha \models s} (-1)^{s-\ell(\alpha)} \fS_{1^{n-s}} = 0~.\qedhere\]
\end{proof}

This allows us to give a right Pieri rule for $\fS_{1^r}$.

\begin{Proposition} \label{prop:ePieri}  For $\alpha$ a composition and $s\geq0$,
 \[\fS_\alpha \fS_{1^s} = \sum_{\substack{\beta \models |\alpha|+s\\\alpha_i\leq\beta_i\leq \alpha_i+1}} \fS_\beta\] where we use the convention that $\alpha_i = 0$ for $i>\ell(\alpha)$.
\end{Proposition}

\begin{proof}
By combining results in Lemma \ref{lemma:perpcomult} and Corollary \ref{cor:equaltoEn}, we
know that
\begin{align*}
F_{1^i}^\perp( F \fS_{1^s} ) -  F_{1^{i-1}}^\perp( F \fS_{1^{s-1}} )
&= \sum_{j=0}^i F_{1^j}^\perp(F) F_{1^{i-j}}^\perp(\fS_{1^{s}}) - \sum_{j=0}^{i-1} F_{1^j}^\perp(F) F_{1^{i-j-1}}^\perp(\fS_{1^{s-1}})\\
&= F_{1^i}^\perp(F) \fS_{1^{s}}
\end{align*}
Hence, we calculate directly that
\begin{align}
\BB_m( F )\fS_{1^s} &=\sum_{i\geq 0} (-1)^i \HH_{m+i} F_{1^i}^\perp( F)\fS_{1^s}\nonumber\\
&=\sum_{i\geq 0} (-1)^i \HH_{m+i} F_{1^i}^\perp( F \fS_{1^s} ) - 
\sum_{i\geq 0} (-1)^i \HH_{m+i} F_{1^{i-1}}^\perp( F \fS_{1^{s-1}} )\nonumber\\
&=\sum_{i\geq 0} (-1)^i \HH_{m+i} F_{1^i}^\perp( F \fS_{1^s} ) - 
\sum_{i\geq 0} (-1)^{i+1} \HH_{m+i+1} F_{1^{i}}^\perp( F \fS_{1^{s-1}} )\nonumber\\
&=\BB_m( F \fS_{1^s})  + \BB_{m+1}( F \fS_{1^{s-1}})~.\label{eq:eonestep}
\end{align}
From this identity a proof of the right-Pieri rule follows by a straight-forward 
induction on the length of the composition $\alpha$ which we leave to the reader.
\end{proof}

\subsection{Pieri rules for skew operators}

The following development contains expressions that include integer tuples (as opposed to compositions).
To ensure that the sets of integer tuples that we consider are finite, we note that it is not necessary
to consider more than a finite set once the degree of the element is fixed.
\begin{Lemma}\label{cor:zeroelts}
For $\alpha \in \ZZ^m$, if $\alpha_i<i-m$ for some $1 \leq i \leq m$, then
$\fS_\alpha = 0$.
Also, if $\sum_{i=1}^m \alpha_i <0$, then $\fS_\alpha = 0$.
\end{Lemma}

\begin{proof}
    By Theorem~\ref{thm:JT},
    $\fS_\alpha  = \sum_{\sigma \in S_m} (-1)^\sigma \HH_{\alpha_1+\sigma_1 -1, \alpha_2 + \sigma_2 -2, \dots, \alpha_m + \sigma_m - m}$.
    Thus, if $\alpha_i < i - m$ for some $i$, then $\alpha_i + \sigma_i
    - i < 0$ for all $\sigma \in S_m$, and hence $H_{\alpha_i + \sigma_i - i}
    = 0$.

    Remark that $\fS_\alpha$ is an element of homogeneous degree $\sum_{i=1}^m
    \alpha_i$, so that if $\sum_{i=1}^m \alpha_i < 0$, then $\fS_\alpha = 0$.
\end{proof}

From the expressions that we have thus far developed, we
are able to describe the action of $F_{1^r}^\perp$ and $F_r^\perp$ 
on the immaculate basis.

\begin{Proposition}\label{prop:eperponfS}
For $r \geq 0$, and for $\alpha \in \ZZ^m$,
\begin{equation}\label{eq:eperponfS}
 F_{1^r}^\perp \fS_{\alpha} = \sum_{\substack{\beta \in \ZZ^{m}\\
\alpha_i - \beta_i \in \{0,1\}\\|\beta|=|\alpha|-r}}
\fS_{\beta}~.
\end{equation}
If $r > \ell(\alpha)$, then $F_{1^r}^\perp \fS_{\alpha}=0$.
\end{Proposition}

\begin{proof}
Using the identities developed in Section \ref{sec:nsymqsymcalc}, we calculate
\begin{align*}
\EEE^\perp_z \BB(z_1) \BB(z_2) \cdots \BB(z_m)\, 1 
&= \EEE^\perp_z \OM_{z_1,z_2,\ldots,z_m} \prod_{1\leq i<j\leq m} (1-z_j/z_i)\\
&= \OM_{z_1,z_2,\ldots,z_m} \prod_{1\leq i<j\leq m} (1-z_j/z_i)\prod_{i=1}^m (1 + z z_i)\\
&=\BB(z_1) \BB(z_2) \cdots \BB(z_m)\, 1 \prod_{i=1}^m (1 + z z_i)\\
&=\BB(z_1) \BB(z_2) \cdots \BB(z_m)\, 1 \sum_{S \subseteq \{ 1,2,\ldots,m \}} 
z^{|S|} \prod_{i \in S} z_i~.
\end{align*}
Now by taking the coefficient of $z^r z_1^{\alpha_1} z_2^{\alpha_2} \cdots z_m^{\alpha_m}$
in both sides of the equation, on the left we have $F_{1^r}^\perp \fS_{\alpha}$, and
on the right we have $\sum_{\substack{S \subseteq \{ 1,2,\ldots, m\}\\ |S| = r}}
\fS_{\alpha_1 - \delta_{1 \in S}, \alpha_2 - \delta_{2 \in S}, \ldots,
\alpha_m - \delta_{m \in S}}$ (where $\delta_{\true} = 1$ and $\delta_{\false}=0$),
which is equivalent to the right hand side of
Equation \eqref{eq:eperponfS}.
\end{proof}

\begin{Proposition}\label{prop:hperponfS}
For $r\geq0$ and for $\alpha \in \ZZ^m$,
\begin{equation}
F_r^\perp \fS_\alpha = \sum_{\substack{\beta \in \ZZ^{m}\\
i-m\leq\beta_i\leq\alpha_i\\|\beta|=|\alpha|-r}}
\fS_\beta~.
\end{equation}
\end{Proposition}

\begin{proof} Again, using identities from Section \ref{sec:nsymqsymcalc},
\begin{align*}
\HHH^\perp_z \BB(z_1) \BB(z_2) \cdots \BB(z_m)\, 1 
&= \HHH^\perp_z \OM_{z_1,z_2,\ldots,z_m} \prod_{1\leq i<j\leq m} (1-z_j/z_i)\\
&= \BB(z_1) \BB(z_2) \cdots \BB(z_m)\, 1 \prod_{i=1}^m 1/(1 - z z_i)\\
&=  \sum_{r \geq 0} z^r \sum_{\substack{\gamma \in \NN^m\\
|\gamma|=r}} z_1^{\gamma_1} z_2^{\gamma_2} \cdots z_m^{\gamma_m}
\BB(z_1) \BB(z_2) \cdots \BB(z_m)\, 1
\end{align*}
If we take the coefficient of $z^r z_1^{\alpha_1} z_2^{\alpha_2} \cdots z_m^{\alpha_m}$
and let $\beta_i = \alpha_i-\gamma_i$, then we have that
$$F_r^\perp \fS_\alpha = \sum_{\substack{\beta \in \ZZ^{m}\\
\beta_i\leq\alpha_i\\|\beta|=|\alpha|-r}}
\fS_\beta~.$$
By Lemma \ref{cor:zeroelts} we may restrict our attention to the $\beta \in \ZZ^m$ such
that $\beta_i \geq i-m$.
\end{proof}

\subsection{The dual immaculate basis}

Every basis $X_\alpha$ of $\Nsym_n$ gives rise to a basis $Y_\beta$ of $\Qsym_n$ defined by duality; 
$Y_\beta$ is the unique basis satisfying $  \langle X_\alpha, Y_\beta \rangle = \delta_{\alpha,\beta}$. 
The dual basis to the immaculate basis of $\Nsym$, denoted $\fS_\alpha^*$, have positive expansions in 
the monomial and fundamental bases of $\Qsym$. 
Furthermore, Corollary \ref{thm:projection} allows us to give
an expansion of the usual Schur functions of $\sym$ (Theorem \ref{thm:decompose}) in terms of these elements. 
All of these results are dual statements to the statements earlier in this section and follow since for a quasi-symmetric
function $G \in \Qsym$,
\[ G = \sum_{\beta} \langle H_\beta, G \rangle M_\beta = \sum_\beta \langle R_\beta, G \rangle F_\beta~.\]

\begin{Proposition}\label{prop:MonomialPositive}
The dual immaculate functions $\fS_\alpha^*$ are monomial positive. Specifically they expand as 
\[ \fS_\alpha^* = \sum_{\beta\leq_\ell\alpha} K_{\alpha, \beta} M_\beta.\]
\end{Proposition}
\begin{proof}
This statement will follow from Proposition \ref{cor:Hexpansion} and duality. Specifically, the coefficient of $M_\beta$ in $\fS_\alpha^*$ is equal to \[\left\langle \HH_\beta ,\fS_\alpha^* \right\rangle = \left\langle \sum_\gamma K_{\gamma, \beta} \fS_\gamma, \fS_\alpha^* \right\rangle = K_{\alpha, \beta}\qedhere\]
\end{proof}

\begin{Proposition}\label{prop:FundamentalPositive}
The dual immaculate functions $\fS_\alpha^*$ are fundamental positive. Specifically they expand as 
\[ \fS_\alpha^* = \sum_{\beta\leq_\ell\alpha} L_{\alpha, \beta} F_\beta.\]
\end{Proposition}
\begin{proof}
From Theorem \ref{thm:Rpositive}, we have
$\left\langle R_\beta, \fS_\alpha^\ast \right\rangle = L_{\alpha\beta}~.$
By Lemma \ref{lemma:bijection}, $L_{\alpha\beta} = 0$ unless $\alpha \geq_\ell\beta$.
\end{proof}

Duality will also yield an explicit expansion of Schur functions into the dual immaculate basis.

\begin{Theorem}\label{thm:decompose}
The Schur function $s_\lambda$, with $\ell(\lambda) = k$ expands into the dual immaculate basis as follows:
\[ s_\lambda = \sum_{\sigma \in S_k} (-1)^\sigma 
\fS_{\lambda_{\sigma_1}+1-\sigma_1, \lambda_{\sigma_2}+2-\sigma_2, \cdots,\lambda_{\sigma_k}+k-\sigma_k}^*\]
where the sum is over permutations $\sigma$ such that $\lambda_{\sigma_i}+i-\sigma_i>0$
for all $i \in \{1,2,\ldots,k\}$.
\end{Theorem}
\begin{proof}
Suppose $s_\lambda = \sum_\alpha c_\alpha \fS_\alpha^*$. Then \[ c_\beta = \left\langle \fS_\beta, \sum_\alpha c_\alpha \fS_\alpha^* \right\rangle = \left\langle \fS_\beta, s_\lambda \right\rangle = \left\langle \chi(\fS_\beta), s_\lambda \right\rangle = \left\langle s_\beta, s_\lambda \right\rangle = \frak{x}_{\beta, \lambda},\] where $\frak{x}_{\beta, \lambda}$ is $(-1)^\sigma$ or $0$ according to the conditions in Proposition \ref{prop:schurcomposition}.
\end{proof}

\begin{Example}
Let $\lambda = [2,2,2,1]$. Then $s_\lambda \in \sym \subseteq \Qsym$ can be expanded in the basis $\{ \fS^*_\alpha \}_\alpha$:
\[s_{2221} =  \fS^*_{2221} - \fS^*_{1321} - \fS^*_{2131} + \fS^*_{1141},\]
since only the permutations $\sigma \in \{ 1234, 2134, 1324, 2314 \}$ contribute to the sum in the expansion of $s_{2221}$.
There are potentially 24 terms in this sum, but for the partition $[2,2,2,1]$ it is easy to reason that $\sigma_4=4$ and $\sigma_1<3$.

These combinatorics arise in the paper of Egge, Loehr and Warrington \cite{ELW} when they describe how to obtain a Schur expansion given a quasi-symmetric fundamental expansion. In their language, these are called ``special rim hook tableau''.

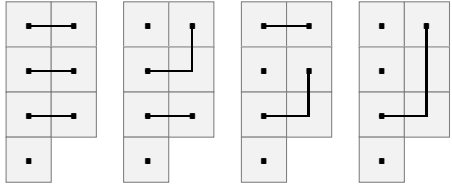
\[
\tikztableau{{ X, X},{X, X},{X,X},{X}}
\tikztableau{{ X, X},{X, X},{X,X},{X}}
\tikztableau{{ X, X},{X, X},{X,X},{X}}
\tikztableau{{ X, X},{X, X},{X,X},{X}}
\put(-165,30){\line(1,0){18}}
\put(-165,13){\line(1,0){18}}
\put(-165,-4){\line(1,0){18}}
\put(-166,-5){$\centerdot$}
\put(-149,-5){$\centerdot$}
\put(-149,12){$\centerdot$}
\put(-166,12){$\centerdot$}
\put(-149,29){$\centerdot$}
\put(-166,29){$\centerdot$}
\put(-166,-22){$\centerdot$}
\put(-120,13){\line(1,0){17.5}}
\put(-120,-4){\line(1,0){18}}
\put(-102.5,13){\line(0,1){18}}
\put(-121,-5){$\centerdot$}
\put(-121,12){$\centerdot$}
\put(-121,29){$\centerdot$}
\put(-121,-22){$\centerdot$}
\put(-104,29){$\centerdot$}
\put(-104,-5){$\centerdot$}
\put(-75,30){\line(1,0){17}}
\put(-75,-4){\line(1,0){16.5}}
\put(-58.5,-4){\line(0,1){18}}
\put(-77,-5){$\centerdot$}
\put(-77,12){$\centerdot$}
\put(-77,29){$\centerdot$}
\put(-77,-22){$\centerdot$}
\put(-60,29){$\centerdot$}
\put(-60,12){$\centerdot$}
\put(-31.5,-4){\line(1,0){17.5}}
\put(-14,-4){\line(0,1){35}}
\put(-32.5,-5){$\centerdot$}
\put(-32.5,12){$\centerdot$}
\put(-32.5,29){$\centerdot$}
\put(-32.5,-22){$\centerdot$}
\put(-15.5,29){$\centerdot$}
\]
\end{Example}

Recall that by Proposition \ref{prop:schurcomposition} we know that
precisely one term in the right hand side of the expansion of Theorem \ref{thm:decompose} is equal to a partition
and we have the following procedure for going from the expansion of a symmetric function $F$ in the $\fS^\ast_\alpha$
basis to the Schur expansion.

\begin{Corollary}
If $F$ is symmetric, and the $\fS^*$-expansion of $F$ is $\sum c_\alpha \fS^*_\alpha$, then the Schur expansion of $F$ is
\[
F = \sum_{\lambda \vdash n} c_{\lambda} s_{\lambda},
\]
where the second sum is taken over all partitions $\lambda$.
\end{Corollary}

\begin{Example}
By Proposition \ref{cor:Hexpansion}, we have that the expansion of
\[
h_{22} = \fS^*_{22} - \fS^*_{13} + \fS^*_{31} + \fS^*_4.
\]
To recover the Schur expansion of $h_{22}$ in terms of Schur functions we can throw away all terms not indexed by partitions
and then $h_{22} = s_{22} + s_{31} + s_4$.
\end{Example}

\subsubsection{Dual immaculate Pieri conjecture}

Similar to the left Pieri rule of the immaculate basis, we conjecture that the dual immaculate basis has a multiplicity free signed Pieri rule. Explicitly \[F_i \fS_\alpha^* = \sum_\beta (-1)^{sign(\alpha, \beta)}\fS_\beta^*,\] for some collection of $\beta$ and some statistic $sign$. 

\begin{Example} We let $\alpha = [2,1,2]$ and $i=2$. Then:
\[F_2 \fS^*_{212} = - \fS^*_{1312} -
 \fS^*_{142} +
 \fS^*_{2212}+
 \fS^*_{3112}+
 \fS^*_{322}+
 \fS^*_{412}.\]
\end{Example}

\subsection{The product of immaculate functions}

In general, the product of two immaculate functions does not expand positively in the immaculate basis. However, for certain products we have a positive expansion.

\begin{Conjecture}\label{conj:positiveproduct}\cite{BBSSZ2}
If $\lambda$ is a partition, then the coefficients $\ccc_{\alpha,\lambda}^\beta$ appearing in \[ \fS_\alpha \fS_\lambda = \sum_\beta \ccc_{\alpha, \lambda}^\beta \fS_\beta,\] are non-negative integers.
\end{Conjecture}

\begin{Example} We give an example of Theorem \ref{conj:positiveproduct} with $\alpha = [1,2]$ and $\lambda = [3,1]$.
\[\fS_{12}\fS_{31}=
\fS_{1231} + \fS_{1321} + \fS_{133} + \fS_{1411} + \fS_{142} + \fS_{151} + \fS_{2221} + \fS_{223}\]\[ + \fS_{2311} + 2\fS_{232} + 2\fS_{241} + \fS_{25} + \fS_{3211} + \fS_{322} + 2\fS_{331} + \fS_{34} + \fS_{421} + \fS_{43}\]
\end{Example}

\begin{Example} We need not look very far to find mixed negative signs in a product of two immaculate functions when
the right one is not indexed by a partition.  For instance
\[ \fS_{1} \fS_{13} = \fS_{113} - \fS_{221} - \fS_{32}~.\]
But products of two immaculate functions indexed by two compositions can potentially be much more complicated. For instance,
\begin{align*} \fS_{11} \fS_{1313} &= \fS_{111313} - \fS_{122113} + \fS_{122221} + \fS_{122232}
-\fS_{13213} - \fS_{212113} + \fS_{212221}\\ 
&+ \fS_{21232} - \fS_{221113} + \fS_{221221} 
+ \fS_{22132}  + \fS_{222121}  - \fS_{22213} + \fS_{222211} \\
& + \fS_{22222}  + \fS_{22231}
- \fS_{23113}  + \fS_{23221}  - \fS_{31213}  - \fS_{32113} + \fS_{32221}  - \fS_{3313} ~.
\end{align*}
The image of the immaculate function indexed by $[1,3,2,1,3]$ under the forgetful map is $\chi(\fS_{13213}) = s_{22222}$. We do not have a rule to predict its coefficient of $-1$ in the above expansion.
This example should indicate how surprising
Theorem \ref{conj:positiveproduct} is given that there are 11 positive terms and 9 negative terms in an example of a
product $\fS_\alpha \fS_\beta$ when the immaculate function $\fS_\beta$ is not indexed by a partition.
\end{Example}

\subsection{The immaculate poset, paths and skew immaculate tableaux}

We create a labelled poset on the set of all compositions, which we call the immaculate poset $\mathfrak{P}$. We place an arrow from $\alpha$ to $\beta$ if $\beta \subset_{1} \alpha$ (equivalently, $\fS_\beta H_1$ expands in the immaculate basis as $\fS_\beta H_1 = \fS_\alpha + \cdots$). Such a cover implies that $\alpha$ and $\beta$ differ by a single box. We give a label of $m$ to this cover, where $m$ is the row containing said box and denote this by $\alpha \xrightarrow{m} \beta$.

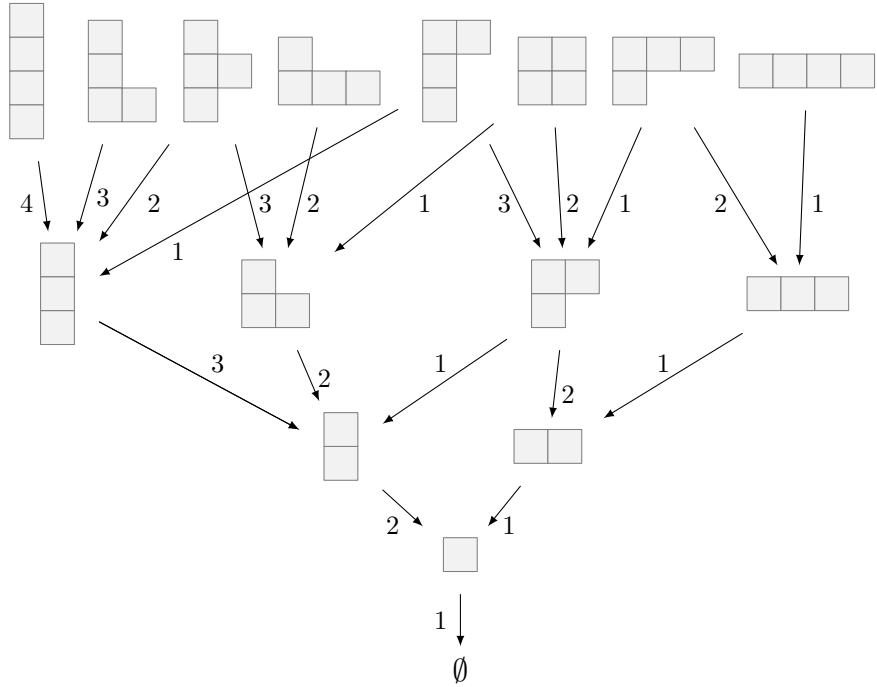
\begin{figure}[htb]
\begin{center}
\begin{tikzpicture}[scale=0.60,>=latex,line join=bevel,]
  \node (0) at (272.36bp,7bp) [draw,draw=none] {$\emptyset$};
  \node (1) at (272.36bp,76bp) [draw,draw=none] {$\tikztableausmall{{X}}$};
  \node (2) at (327.36bp,144bp) [draw,draw=none] {$\tikztableausmall{{X,X}}$};
  \node (11) at (197.36bp,144bp) [draw,draw=none] {$\tikztableausmall{{X}, {X}}$};
  \node (3) at (484.36bp,240bp) [draw,draw=none] {$\tikztableausmall{{X,X,X}}$};
  \node (21) at (338.36bp,240bp) [draw,draw=none] {$\tikztableausmall{{X,X}, {X}}$};
  \node (12) at (156.36bp,240bp) [draw,draw=none] {$\tikztableausmall{{X}, {X, X}}$};
  \node (111) at (19.359bp,240bp) [draw,draw=none] {$\tikztableausmall{{X}, {X}, {X}}$};
   \node (4) at (490bp,380bp) [draw,draw=none] {$\tikztableausmall{{X,X,X, X}}$};
      \node (31) at (400bp,380bp) [draw,draw=none] {$\tikztableausmall{{X,X,X}, {X}}$};
   \node (22) at (330bp,380bp) [draw,draw=none] {$\tikztableausmall{{X,X},{X, X}}$};
   \node (211) at (270bp,380bp) [draw,draw=none] {$\tikztableausmall{{X,X},{X}, {X}}$};
   \node (13) at (190bp,380bp) [draw,draw=none] {$\tikztableausmall{{X},{X,X, X}}$};
   \node (121) at (120bp,380bp) [draw,draw=none] {$\tikztableausmall{{X},{X,X}, {X}}$};
   \node (112) at (60bp,380bp) [draw,draw=none] {$\tikztableausmall{{X},{X},{X, X}}$};
   \node (1111) at (0bp,380bp) [draw,draw=none] {$\tikztableausmall{{X},{X},{X}, {X}}$};
  \definecolor{strokecol}{rgb}{0.0,0.0,0.0};
  \pgfsetstrokecolor{strokecol}
  \draw [<-] (11) -- (21);
  \draw [<-] (11) -- (12);
  \draw [<-] (11) -- (111);
  \draw [<-] (2) -- (21);
  \draw [<-] (1) -- (2);
  \draw [<-] (1) -- (11);
  \draw [<-] (0) -- (1);
  \draw [<-] (2) -- (3);
  \draw [<-] (11) -- (111);
  \draw [<-] (111) -- (1111);
    \draw [<-] (111) -- (112);
      \draw [<-] (111) -- (121);
        \draw [<-] (111) -- (211);
          \draw [<-] (12) -- (121);
            \draw [<-] (12) -- (13);
              \draw [<-] (12) -- (22);
                \draw [<-] (21) -- (211);
                  \draw [<-] (21) -- (22);
                    \draw [<-] (21) -- (31);
                      \draw [<-] (3) -- (31);
                        \draw [<-] (3) -- (4);
\draw (260bp,38bp) node {\footnotesize $1$};

\draw (303bp,98bp) node {\footnotesize $1$};
\draw (230bp,98bp) node {\footnotesize $2$};

\draw (260bp,200bp) node {\footnotesize $1$};
\draw (120bp,200bp) node {\footnotesize $3$};
\draw (187bp,190bp) node {\footnotesize $2$};
\draw (340bp,180bp) node {\footnotesize $2$};
\draw (400bp,200bp) node {\footnotesize $1$};

\draw (0bp,300bp) node {\footnotesize $4$};
\draw (48bp,304bp) node {\footnotesize $3$};
\draw (80bp,300bp) node {\footnotesize $2$};
\draw (95bp,270bp) node {\footnotesize $1$};
\draw (150bp,300bp) node {\footnotesize $3$};
\draw (180bp,300bp) node {\footnotesize $2$};
\draw (250bp,300bp) node {\footnotesize $1$};
\draw (300bp,300bp) node {\footnotesize $3$};
\draw (343bp,300bp) node {\footnotesize $2$};
\draw (376bp,300bp) node {\footnotesize $1$};
\draw (436bp,300bp) node {\footnotesize $2$};
\draw (497bp,300bp) node {\footnotesize $1$};
\end{tikzpicture}
\end{center}
\caption{The first few levels of $\mathfrak{P}$}
\label{fig:immaculateposet}
\end{figure}

Maximal chains on this poset from $\alpha$ to $\emptyset$ are equivalent to
immaculate standard tableaux and maximal chains on an interval from
$\alpha$ to $\beta$ are what we will call skew immaculate tableaux of shape $\alpha/\beta$.  
We can visualize a path 
$\{\alpha = \beta^{(0)} \xrightarrow{m_1} \beta^{(1)} \xrightarrow{m_2} \cdots \xrightarrow{m_k} \beta^{(k)} = \beta\}$ 
as a labelled composition diagram of outer shape $\alpha$ by
labeling a cell in row $m_i$ with a $k-i+1$.  The cells representing $\beta$ will not have a label and the labels must increase from
left to right in a row.

\begin{Example} Consider the example of a skew immaculate tableau of shape $[2,2,2]/[1,2]$ given by the diagram below.  
We show the labelled composition on the left and the corresponding representation as a path to the right. 
$$\tikztableausmall{{,2},{,}, {1,3}}~\longleftrightarrow~\{[2,2,2] \xrightarrow{3} [2,2,1] \xrightarrow{1} [1,2,1] \xrightarrow{3} [1,2]\}~.$$
\end{Example}

A path $P = \{\alpha = \beta^{(0)} \xrightarrow{m_1} \beta^{(1)} \xrightarrow{m_2} 
\cdots \xrightarrow{m_k} \beta^{(k)} = \beta\}$ in $\mathfrak{P}$ 
will be called a horizontal $k$-strip if $m_1 \leq m_2 \leq \cdots \leq m_k$.  
The Pieri rule for the immaculate functions from Theorem \ref{thm:Pieri} says that there is
a horizontal $k$-strip from $\alpha$ to $\beta$ if and only if
$\fS_{\alpha}$ appears in the expansion of $\fS_{\beta} H_k$.

The \emph{descent composition} of a word $\ell_1, \ell_2, \dots, \ell_n$ is
the composition
$[
i_1,
i_2 - i_{1},
\ldots,
i_j - i_{j-1},
n - i_j
]$,
where $i_1 < i_2 < \dots < i_j$ are the descents of the word; that is, the elements in $\{1, \dots, n-1\}$ such
that $\ell_{i_a} > \ell_{i_a+1}$. For example, the descent composition of the word $3,2,1,1,4,1$ is $[1,1,3,1]$.
For a path $P = \{\beta_0 \xrightarrow{m_1} \beta_1 \xrightarrow{m_2} \cdots \xrightarrow{m_n} \beta_n\}$, 
the descent composition $\alpha(D(P))$ associated to $P$ is the reverse of the descent composition of the word 
$m_1, m_2, \dots, m_n$.  When $P$ is a standard tableau of shape $[\emptyset, \alpha]$, this notion of descent
set is equivalent to the definition of descent from Definition \ref{def:descentSIT}.
Equivalently, the descent composition describes how $P$ can be decomposed into horizontal 
$k$-strips of maximal lengths.

\subsection{Skew dual immaculate quasi-symmetric functions}

Following notions in \cite{BMSW}, since there is a positive right Pieri rule on the
immaculate basis of $\Nsym$,  there is a natural way of constructing skew dual immaculate
elements of $\Qsym$. 
For an interval $\{ \gamma : \beta \subseteq \gamma \subseteq \alpha \} \subseteq \mathfrak{P}$, we then define the skew dual 
immaculate function as
\[\fS_{\alpha/\beta}^* = \sum_{\gamma} \langle \fS_\beta H_\gamma, \fS_\alpha^\ast \rangle M_\gamma~. \]
By Theorem \ref{thm:Pieri} and the notion of standardization from Definition \ref{def:standardization},
the coefficient $\langle \fS_\alpha H_\gamma, \fS_\alpha^\ast \rangle$ is equal
to the number of skew standard immaculate tableaux of shape $\alpha/\beta$
with descent composition coarser than $\gamma$.

Right multiplication by $H_r$ on an element $\fS_\beta$ is a {\it Pieri operator}
\cite{BMSW} on the poset of compositions
defined in the previous section.  By Theorem 2.3 of \cite{BMSW} we have an expansion of
these elements into other bases by the use of the dual pairing between $\Nsym$ and $\Qsym$.
\begin{Proposition} \label{prop:fundDimmexpan}  For $\{ \gamma : \beta \subseteq \gamma \subseteq \alpha \}$ 
an interval of $\mathfrak{P}$,
\[ \fS_{\alpha/\beta}^* =\sum_\gamma
\langle \fS_\beta R_\gamma, \fS_\alpha^\ast \rangle F_\gamma = 
\sum_\gamma \langle \fS_\beta \fS_\gamma, \fS_\alpha^\ast \rangle \fS_\gamma^\ast~. \]
\end{Proposition}

The coefficients $\ccc^{\alpha}_{\beta,\gamma} = \langle \fS_\beta \fS_\gamma, \fS_\alpha^\ast \rangle$ are
those that appear in the expansion
\[ \fS_\beta \fS_\gamma = \sum_\alpha \ccc_{\beta,\gamma}^\alpha \fS_\alpha~. \]
Theorem \ref{conj:positiveproduct} states that certain of these coefficients will be positive,
but they aren't positive in general.

The expansion of these elements in the fundamental and monomial bases are positive.
An argument similar to that given for Theorem \ref{thm:Rpositive} and 
Proposition \ref{prop:FundamentalPositive} shows that
the skew immaculate tableaux can be used to give a combinatorial expansion of
these elements in the fundamental basis.
\begin{Proposition} \label{prop:fundamentalexpansion}
Let $\{ \gamma : \beta \subseteq \gamma \subseteq \alpha \}$ be an interval of $\mathfrak{P}$, then
\[ \fS_{\alpha/\beta}^\ast = \sum_P F_{\alpha(D(P))}\]
where the sum is over all paths $P$ in $\mathfrak{P}$ from $\alpha$ to $\beta$ (alternatively, skew standard
immaculate tableaux of shape $\alpha/\beta$).
\end{Proposition}

\begin{proof} Let $K_{\alpha/\beta,\gamma} = \langle \fS_\beta H_\gamma, \fS_\alpha^\ast \rangle$ which
we have already noted is equal to the number of skew immaculate tableaux of shape $\alpha/\beta$
whose descent composition is coarser than~$\gamma$.  Let $L_{\alpha/\beta, \tau}$ to be the number of skew immaculate tableaux of shape
$\alpha/\beta$ whose descent composition is equal to $\tau$.  Clearly we have 
$K_{\alpha/\beta,\gamma} = \sum_{\tau \geq \gamma} L_{\alpha/\beta, \tau}$ and by M\"obius inversion
this is equivalent to $\sum_{\gamma \geq \tau} (-1)^{\ell(\tau)-\ell(\gamma)} K_{\alpha/\beta,\gamma} 
= L_{\alpha/\beta,\tau}$~.
\begin{align*}
\langle R_\tau, \fS_{\alpha/\beta}^\ast \rangle &=
\langle \fS_\beta R_\tau, \fS_\alpha^\ast \rangle \\
&= \sum_{\gamma \geq \tau} (-1)^{\ell(\tau)-\ell(\gamma)} \langle \fS_\beta H_\gamma, \fS_\alpha^\ast \rangle\\
&= \sum_{\gamma \geq \tau} (-1)^{\ell(\tau)-\ell(\gamma)} K_{\alpha/\beta,\gamma} = L_{\alpha/\beta,\tau}~.\qedhere
\end{align*}
\end{proof}

\begin{Example}
There are $6$ paths from $[1,3,2]$ to $[1,1]$.  We give their representation as a tableau diagram in the table below.
Below each of the diagrams we also give the descent composition of the path.
\begin{center}
\begin{tabular}{cccccc}
$\tikztableausmall{{X},{X,1,2}, {3,4}}$&
$\tikztableausmall{{X},{X,1,3}, {2,4}}$&
$\tikztableausmall{{X},{X,1,4}, {2,3}}$&
$\tikztableausmall{{X},{X,2,3}, {1,4}}$&
$\tikztableausmall{{X},{X,2,4}, {1,3}}$&
$\tikztableausmall{{X},{X,3,4}, {1,2}}$\cr
$[4]$&$[2,2]$&$[3,1]$&$[1,3]$&$[1,2,1]$&$[2,2]$
\end{tabular}
\end{center}
The fundamental expansion of this element is positive and is the sum of $6$ terms corresponding to
these tableaux, but the expansion in the dual immaculate basis is not positive in this case.
\[\fS_{132/11}^\ast = F_{121} + F_{13} + 2 F_{22} + F_{31} + F_4 = -\fS_{13}^\ast + \fS_{22}^\ast + \fS_{31}^\ast + \fS_4^\ast~.\] 
\end{Example}

\section{Hall-Littlewood basis for $\Nsym$}

As an application and demonstration of the power of the immaculate basis, 
we use this section to build lifts of Hall-Littlewood functions inside 
$\Nsym$. We begin by reminding the reader of a definition of Hall-Littlewood 
symmetric functions inside $\sym$ defined by creation (or vertex) operators.

\subsection{Hall-Littlewood symmetric functions}

The Hall-Littlewood symmetric functions $P_\lambda$ were first studied by Hall. 
They are symmetric functions with a parameter $q$ (i.e. elements of 
$\sym[q] := \mathbb{Q}(q)[h_1, h_2, \dots]$). They generically ($q$ not 
specialized) form a basis for $\sym[q]$. When $q=0$, they specialize to 
monomial symmetric functions and at $q = 1$ they specialize to Schur functions. 
We are interested in the corresponding dual basis 
$Q'_\mu$ which have the property that
$\langle P_\lambda, Q'_\mu \rangle = \delta_{\lambda, \mu}$.
The elements $Q'_\mu$ have the expansion in the Schur basis
\[ Q'_\mu = \sum_\lambda K_{\lambda\mu}(q) s_\lambda\]
where $K_{\lambda\mu}(q)$ are the $q$-Kostka
polynomials.  At $q=0$ the $Q'_\lambda$ specialize to Schur functions $s_\lambda$
and at $q=1$ they specialize to homogeneous complete functions $h_\lambda$.

We define an operator $\widetilde{\B}_m: \sym[q]_n \to \sym[q]_{n+m}$ by: \[ \widetilde{\B}_m := \sum_{i\geq0} q^i \B_{m+i} h_i^\perp.\]

\begin{Theorem}\label{jing} (Jing \cite{Jing}, see also \cite{Mac} p. 237-8)
If $m \geq \lambda_1$, then \[ \widetilde{\B}_m Q'_\lambda = Q'_{(m, \lambda)}~.\]
\end{Theorem}

\begin{Definition} \label{def:Qpsym}
One may define, for any integer tuple $\alpha \in {\mathbb Z}^m$, a symmetric function 
\[Q'_\alpha := \widetilde{\B}_{\alpha_1} \cdots \widetilde{\B}_{\alpha_m}(1)~.\]
\end{Definition}

\begin{Remark}  The $\widetilde{\B}_m$ operators generalize the $\B_m$ operators of Bernstein \cite{Zel}
and are due to Jing \cite{Jing} who studied them as a `vertex operator' definition
of the algebra of Hall-Littlewood polynomials $Q_\lambda$.
Garsia \cite{G} used a modified version as creation operators for the symmetric
functions $Q'_\lambda$.  Specializations of these operators can be used to
create $P$ and $Q$-Schur functions (see for instance \cite{HH,J2,Mac}). The commutation relations of these
operators make it natural to consider the symmetric functions for all compositions or integer tuples,
but only the $Q'_\lambda$ for $\lambda$ a partition are known to be positive when expanded in
the Schur basis. Generalizations were considered for Macdonald symmetric functions \cite{LV, KN} as well
as a technique for proving the polynomiality of the Macdonald-Kostka coefficients.
  Shimozono and Zabrocki \cite{SZ} considered compositions of these operators
which were indexed by tuples of integers.  These functions were recently studied
further in \cite{HMZ} where it was shown that 
\begin{equation}\label{eq:composum}
\sum_{\alpha \models n} (-q)^{n-\ell(\alpha)} Q'_\alpha = e_n[X]~.
\end{equation}
In that case, Hall-Littlewood functions indexed by compositions were used to
understand the action of the operator $\nabla$ introduced in \cite{BGHT} on a spanning set of the symmetric functions.
\end{Remark}

\subsection{A new Hall-Littlewood basis for $\Nsym$}

We start by building operators $\widetilde{\BB}_m:\Nsym[q] \rightarrow \Nsym[q]$, defined 
for $m \in \ZZ$ by 
\begin{equation}
\label{HLopdef} \widetilde{\BB}_{m} (f) = \sum_{i \geq 0} q^i \BB_{m+i} F_i^\perp (f)~.
\end{equation}

We may now define our new basis of $\Nsym$; they are the result of applying successive $\widetilde{\BB}_m$ operators.

\begin{Definition} \label{def:NCHL}
If $\alpha = [\alpha_1, \alpha_2, \dots, \alpha_m] \in \ZZ^m$, then we define 
\[\QQ'_\alpha = \widetilde{\BB}_{\alpha_1}\widetilde{\BB}_{\alpha_2} \cdots \widetilde{\BB}_{\alpha_m}(1)~.\]
\end{Definition}

\begin{Example}
We will calculate $\QQ'_{4,2}$ and expand this in the $\fS$-basis. 
First note that $\QQ'_{2} = \widetilde{\BB}_2 (1) = \BB_2(1) = \fS_{2}$. 
Next we apply $\widetilde{\BB}_4$. 
\[ \QQ'_{42} = \widetilde{\BB}_4 \fS_{2} 
= \BB_4 \fS_{2} + q \BB_5 F_1^\perp (\fS_{2}) + q^2 \BB_6 F_2^\perp(\fS_{2}) 
= \fS_{42} + q\fS_{51} + q^2\fS_{6}~.\]
\end{Example}

It is worth presenting a further example which seems a convincing reason why the immaculate
basis and the basis $\QQ_\alpha'$
merit further study.  By definition, we have the property that
$\chi( \QQ_\alpha' ) = Q_\alpha'$.  We know that for a partition $\lambda$, the $Q_\lambda'$ form a
basis of $\sym$ and are Schur positive.  It seems (see Conjecture \ref{conj:positivity}) that the elements $\QQ_\lambda'$
are immaculate positive.  What is surprising is that
there are more terms in the expansion of $\QQ_\lambda' \in \Nsym$ than in the expansion of $Q_\lambda' \in \sym$
and so through surjection under $\chi$ many of the positive terms of $\QQ_\lambda'$ cancel.

\begin{Example}  The basis element $\QQ_{1111}' \in \Nsym$ has the following expansion in the immaculate
basis
\begin{align*} \QQ_{1111}' = \fS_{1111} &+ q \fS_{112} + (q+q^2)\fS_{121} + q^3\fS_{13}+ (q+q^2+q^3) \fS_{211}\\ &+
(q^2+q^3+q^4) \fS_{22} + (q^3+q^4+q^5) \fS_{31} + q^6 \fS_4~.
 \end{align*}
The image of $\chi$ on this element is the Hall-Littlewood symmetric function
\[ Q_{1111}' = s_{1111} + (q+q^2+q^3) s_{211} + (q^2+q^4) s_{22} + (q^3+q^4+q^5) s_{31} + q^6 s_4~.\]
For any partition $\lambda$, $\chi(\fS_\lambda) = s_\lambda$.  For the other immaculate basis elements of $\Nsym$, 
$\chi(\fS_{112}) = \chi(\fS_{121}) = 0$
and $\chi(\fS_{13}) = -s_{22}$.
\end{Example}

\subsection{A right-Pieri type rule on the non-commutative Hall-Littlewood basis}

We start this section by developing a Pieri rule for the $\QQ'$-basis of $\Nsym[q]$. 

In the same way that we derived Lemma \ref{lemma:Bexpansion} from Lemma \ref{lemma:Frule}, we can also derive the following result.

\begin{Lemma}\label{lemma:Btilda} For $f \in \Nsym$ and $s, m \geq 0$, 
\[ \widetilde{\BB}_m (f \HH_s) = \sum_{k=0}^s q^k \widetilde{\BB}_{m+k} (f) \HH_{s-k} 
- \sum_{k=0}^{s-1} q^k \widetilde{\BB}_{m+k+1} (f) \HH_{s-k-1}~.
 \]
\end{Lemma}
\begin{proof} We explicitly calculate the left hand side and apply Lemma \ref{lemma:Frule} (where $F_j=0$ for $j<0$.
\begin{align*}
\widetilde{\BB}_m (f \HH_s) &=  \sum_{i \geq 0} q^i \BB_{m+i} F_i^\perp (f \HH_s)
= \sum_{i \geq 0} q^i \BB_{m+i}\left(\sum_{k=0}^s F_{i-k}^\perp(f) \HH_{s-k}\right)\\
&= \sum_{k=0}^s \sum_{i \geq 0} q^i \BB_{m+i} \left( (F_{i-k}^\perp(f)) \HH_{s-k} \right)~.
\end{align*}
Now by Lemma \ref{lemma:Bexpansion} this last expression is equal to 
\[= \sum_{k=0}^s \sum_{i \geq 0} q^i \left( \BB_{m+i} F_{i-k}^\perp(f) \HH_{s-k} - \BB_{m+i+1} F_{i-k}^\perp(f) \HH_{s-k-1} \right)~.\]
If we let $j=i-k$ and shift the indices of the sum, we have
\begin{align*}
&= \sum_{k=0}^s q^k \sum_{j \geq 0} q^{j} \left( \BB_{m+k+j} F_{j}^\perp(f) \HH_{s-k} - \BB_{m+k+1+j} F_{j}^\perp(f) \HH_{s-k-1} \right)\\
&= \sum_{k=0}^s q^k \widetilde{\BB}_{m+k} (f) \HH_{s-k} - \sum_{k=0}^{s-1} q^k \widetilde{\BB}_{m+k+1} (f) \HH_{s-k-1}~.
\end{align*}
The second sum in the last equation terminates at $k=s-1$ since for $k=s$, 
$\HH_{s-k-1} = \HH_{-1} = 0$.
\end{proof}

For
 two compositions $\alpha, \beta$, with $\alpha_i \leq \beta_i$ for all $i$, we let $n(\alpha, \beta)$ denote the number of rows of $\alpha$ which are strictly shorter than the same row of $\beta$ (i.e. $n(\alpha,\beta) =| \{ i : \alpha_i < \beta_i \}|$).

\begin{Theorem}\label{thm:HLpieri} The $\QQ_\alpha'$ have a right-Pieri type rule given by:
\[ \QQ_\alpha' \HH_s = \sum_{\alpha \subset_{s} \beta} (1-q)^{n(\alpha, \beta)} \QQ_\beta'.\]
\end{Theorem}
\begin{proof}
    We will prove this using a double induction.
First we use induction on $s$. The case $s = 0$ is trivially true as $\HH_0 = 1$. 
We now assume it true for all $j < s$ and prove it for $s$. For a given $s$ we make a second inductive 
hypothesis on the length of the composition $\alpha$. Again, for compositions of length $0$ it is 
trivially true so we assume it to hold for length $\ell \leq \ell(\alpha)$ and prove it for 
$\ell(\alpha)+1$. We apply Lemma \ref{lemma:Btilda} to yield:
\begin{align}
\QQ_{[m, \alpha]}'\HH_s &= [\widetilde{\BB}_m (\QQ_\alpha')] \HH_s\nonumber\\ 
&= \widetilde{\BB}_m(\QQ_\alpha' \HH_s) - \sum_{k = 1}^s q^k \widetilde{\BB}_{m+k}(\QQ_\alpha') \HH_{s-k} + \sum_{k=0}^{s-1} q^k \widetilde{\BB}_{m+k+1}(\QQ_\alpha') \HH_{s-k-1}~.\label{kitten}
\end{align}
When we apply the inductive hypotheses, (\ref{kitten}) becomes:
\begin{equation}\label{catfish} 
=\sum_{\alpha \subset_{s} \beta} (1-q)^{n(\alpha, \beta)} \QQ_{[m, \beta]}' 
 - \sum_{k=1}^s \sum_{[m+k, \alpha] \subset_{s-k} \gamma} q^k (1-q)^{n([m+k, \alpha], \gamma)} \QQ_{\gamma}'
 \end{equation}
 \begin{equation*}
 + \sum_{k=0}^{s-1} \sum_{[m+k+1, \alpha] \subset_{s-k-1} \delta} q^k (1-q)^{n([m+k+1, \alpha], \delta)} \QQ_\delta'.
 \end{equation*}

We need to show that the coefficient of a term $\QQ_\eta'$ is $(1-q)^{n([m,\alpha],\eta)}$. If $\eta = [m, \zeta]$ then the only coefficient in  (\ref{catfish}) will come from the first sum; forcing $\beta = \zeta$ yields $(1-q)^{n(\alpha, \zeta)} = (1-q)^{n([m,\alpha], \eta)}$.

If $\eta = [m+a, \zeta]$ with $s \geq a \geq 1$ then the contribution to $Q_\eta'$ will come from the second and third summands in (\ref{catfish}). The coefficient coming from the second summand is: \[ - \sum_{k=1}^{a} q^k (1-q)^{n( [m+k,\alpha], \eta)} = - \sum_{k=1}^{a-1} q^k (1-q)^{1+n(\alpha, \zeta)} - q^a (1-q)^{n(\alpha, \zeta)} \] \[ = -(1-q)^{n(\alpha, \zeta)} \left( \left( \sum_{k=1}^{a-1} q^k (1-q)\right) + q^a \right)= -q (1-q)^{n(\alpha, \zeta)}.\]
The coefficient coming from the third summand is: 
\[ \sum_{k=0}^{a-1} q^k (1-q)^{n( [m+k+1,\alpha], \eta)} = \sum_{k=0}^{a-2} q^k (1-q)^{1+n(\alpha, \zeta)} - q^{a-1} (1-q)^{n(\alpha, \zeta)} \] \[ = (1-q)^{n(\alpha, \zeta)} \left( \left( \sum_{k=0}^{a-2} q^k (1-q)\right) + q^{a-1} \right)= (1-q)^{n(\alpha, \zeta)}.\]
Combining these coefficients gives $(1-q)^{n(\alpha, \zeta) + 1} = (1-q)^{n([m,\alpha], \eta)}$.
\end{proof}

\begin{Example}
Similar to Example \ref{pigeon}, we compute,
\[
\QQ'_{23} \HH_3 = 
 \QQ'_{233} + 
(1-q) \QQ'_{242} +
(1-q) \QQ'_{251} +
(1-q) \QQ'_{26} +
(1-q) \QQ'_{332} \] \[ +
(1-q)^2 \QQ'_{341} +
(1-q)^2 \QQ'_{35}+
(1-q) \QQ'_{431}+
(1-q)^2 \QQ'_{44}+
(1-q) \QQ'_{53}.
\]
\end{Example}

As an immediate corollary, repeated application of our Pieri rule gives an explicit formula for the expansion of the $\HH$ basis.
For an immaculate tableau $T$, we let $n(T) = \sum_{j = 0}^{\ell(sh(T))} d_j(T)-1,$ where $d_j(T)$ is the number of distinct entries in row $j$. Equivalently, $n(T)$ counts the number of times that a distinct letter $l$ exists in a row which does not start with $l$.

\begin{Theorem}\label{HtoHL} For a composition $\beta$,
\[ \HH_\beta = \sum_T (1-q)^{n(T)} \QQ_{shape(T)}',\] the sum over all immaculate tableaux of content $\beta$.
\end{Theorem}

\begin{proof}
This follows from repeated application of Theorem \ref{thm:HLpieri}.
\end{proof}

\begin{Example}
The five immaculate tableaux of shape $[4,2,3]$ and content $[3,1,2,3]$ from Example \ref{ex:immaculatetableau} have $n(T) = 2, 3, 3, 3, 1$ respectively: 
\[ \tikztableausmall{{1,1, 1, 3},{2, 3}, {4,4,4}} 
\tikztableausmall{{1,1, 1, 3},{2, 4}, {3,4,4}} 
\tikztableausmall{{1,1, 1, 4},{2,3}, {3,4,4}} 
\tikztableausmall{{1,1, 1, 4},{2, 4}, {3,3,4}} 
\tikztableausmall{{1,1, 1, 2},{3, 3}, {4,4,4}} 
\]
Therefore, \[\HH_{3123} = \cdots + \left((1-q) + (1-q)^2 + 3(1-q)^3\right) \QQ'_{423} + \cdots.\]
\end{Example}

\begin{Theorem} For generic $q$,
 the $\QQ_\alpha'$ form a basis of $\Nsym[q]$. When $\alpha$ is a partition, $\chi(\QQ_\alpha') = Q_\alpha'$. Moreover, the $\QQ_\alpha'$ interpolate between the complete homogeneous basis (at specialization $q = 1$) and the immaculate basis (at specialization $q = 0$).
\end{Theorem}

\begin{proof}
It is easy to see that the $\QQ_\alpha'$ expand into the immaculate basis upper unitriangularly with respect to lexicographic ordering. The second statement follows from Theorem \ref{jing} and the fact that $\chi \circ \widetilde{\BB}_m = \widetilde{\B}_m \circ \chi$, which follows from the fact that $\chi \circ \BB_m  = \B_m \circ \chi$, which is easy to check on the immaculate basis. The only immaculate tableau $T$ which has $n(T) = 0$ and content $\beta$ is the tableau with $\beta_i$ many $i$'s in row $i$.
 Substituting $q=1$ in Theorem \ref{HtoHL} then proves that $\QQ'_{\beta}$ is $\HH_\beta$. Finally, substituting $q=0$ into the definition of $\widetilde{\BB}_m$ shows that $\widetilde{\BB}_m = \BB_m$ in this case.
\end{proof}

Equation \ref{HLopdef} defines $\widetilde{\BB}_m$ in terms of the $\BB_n$ operators.  We can also give a similar expression
for the $\BB_n$ operators in terms of $\widetilde{\BB}_m$.

\begin{Proposition}\label{prop:btobtilda}  For $m \in \ZZ$,
\begin{equation*}
\BB_m = \sum_{i \geq 0} (-q)^i \widetilde{\BB}_{m+i} F_{1^i}^\perp~.
\end{equation*}
\end{Proposition}

\begin{proof}
Our calculations below will use the identity that 
$\sum_{i=0}^n (-1)^i F_{n-i} F_{1^i} = 0$ for $n>0$ (see for instance Stanley 
\cite{Sta}, where $F_{n-i} = h_{n-i}$ and $F_{1^i} = e_i$).
We begin with the left hand side of the expression and then group terms together with $i+j=n$
\begin{align*}
\sum_{i \geq 0} (-q)^i \widetilde{\BB}_{m+i} F_{1^i}^\perp 
&=\sum_{i,j \geq 0} (-q)^i q^j \BB_{m+i+j} F_{j}^\perp F_{1^i}^\perp\\
&=\sum_{n \geq 0} \sum_{i=0}^n (-1)^{i} q^{n} \BB_{m+n} F_{n-i}^\perp F_{1^i}^\perp\\
&=\BB_{m}~.\qedhere
\end{align*}
\end{proof}

This formula potentially gives us a recursive means of computing an element of the immaculate basis in
terms of the $\QQ'$-basis.  Although it isn't obvious how this should be done for all compositions, 
for certain compositions there are remarkably simple expressions.
In particular,
the analogue of Equation \eqref{eq:composum} also holds for our lifted Hall-Littlewood 
functions and they are proven here using roughly the analogous proof to that provided in
\cite{HMZ}.

\begin{Proposition} For $n\geq1$,
\[\fS_{1^n} = \sum_{\alpha \models n} (-q)^{n-\ell(\alpha)} \QQ'_\alpha.\]
More generally, for $1 \leq k \leq n$,
\[\fS_{k,1^{n-k}}= \sum_{\substack{\alpha \models n\\\alpha_1\geq k}} 
(-q)^{n-k+1-\ell(\alpha)} \QQ'_\alpha~.\]
\end{Proposition}

\begin{proof} For the first statement we proceed by induction on $n$.
 The base case is that $\fS_n = \HH_n=\QQ'_n$, which has already been established.
 \[ \fS_{1^n} = \BB_1 \fS_{1^{n-1}} = \sum_{i \geq 0} (-q)^i \widetilde{\BB}_{1+i} F_{1^i}^\perp \fS_{1^{n-1}} = \sum_{i \geq 0} (-q)^i \widetilde{\BB}_{1+i} \fS_{1^{n-i-1}}, 
 \]
 since $F_{1^i}^\perp \fS_{1^{n-1}} = \fS_{1^{n-1-i}}$ by Corollary \ref{cor:equaltoEn}. Next we apply the inductive hypothesis:
 \[ 
\sum_{i \geq 0} (-q)^i \widetilde{\BB}_{1+i} \fS_{1^{n-i-1}}
= \sum_{i\geq 0} \sum_{\beta \models n-i-1} (-q)^i (-q)^{n-i-1 - \ell(\beta)} \widetilde{\BB}_{1+i}\QQ'_\beta = \sum_{\alpha \models n} (-q)^{n-\ell(\alpha)} \QQ'_\alpha.\]
 
The second expression follows by applying $\BB_k$ to the expression just derived
and again applying Corollary \ref{cor:equaltoEn}.
 \[ \fS_{k, 1^{n-k}} = \BB_k \fS_{1^{n-k}} = \sum_{i \geq 0} 
 (-q)^i \widetilde{\BB}_{k+i} F_{1^i}^\perp \fS_{1^{n-k}} = 
 \sum_{i \geq 0} (-q)^i \widetilde{\BB}_{k+i} \fS_{1^{n-k-i}} 
 \]
\[ = \sum_{i\geq 0} \sum_{\beta \models n-k-i} 
(-q)^i (-q)^{n-k-i - \ell(\beta)} \widetilde{\BB}_{k+i}\QQ'_\beta 
= \sum_{\substack{\alpha \models n\\\alpha_1\geq k}} 
(-q)^{n-k+1-\ell(\alpha)} \QQ'_\alpha~.\qedhere\]
 
\end{proof}

We let $\{ \mathcal{P}_\alpha\}$ denote the basis of $\Qsym$ which is dual to the $\{ \QQ_\alpha'\}$ basis of $\Nsym$. Then by duality, the $\mathcal{P}_\alpha$ are monomial positive, in the variable $(1-q)$.

\begin{Theorem} For a composition $\alpha$, 
\[\mathcal{P}_\alpha = \sum_T (1-q)^{n(T)} M_{\content(T)} ,\] the sum over all immaculate tableaux of shape $\alpha$. Moreover, $\mathcal{P}_\alpha$ interpolates between the monomial basis of $\Qsym$ (at $q = 1$) and the dual immaculate basis (at $q=0$).
\end{Theorem}

It is a fundamental combinatorial result of Lascoux and Sch\"utzenberger \cite{LS} 
that the function $Q'_\lambda$ expands positively in the Schur basis:
\[ Q'_\lambda = \sum_T q^{charge(T)} s_{shape(T)},\] the sum over all standard 
Young tableaux of content $\lambda$. We end our paper with a similar 
conjectured expansion of our new lifted Hall-Littlewood basis into the immaculate basis.

\begin{Conjecture} \label{conj:positivity}
If $\lambda$ is a partition then $\QQ_\lambda'$ expands in the immaculate basis $\fS_\beta$ with coefficients which are positive polynomials in $q$. More explicitly, \[\QQ_\lambda' = \sum_{T} q^{st(T)} \fS_{shape(T)},\] for some statistic $st$, over all immaculate tableaux of content $\lambda$.
\end{Conjecture}

\begin{Example} For $\lambda = [3,3,1]$, the expansion of $\QQ_\lambda'$ corresponds to the number of immaculate tableaux of content $\lambda$.
$$
\begin{array}{cccccc}
&
\tikztableausmall{{1,1,  1},{2,  2,2},  {3}}  &
\tikztableausmall{{1,1,  1},{2,  2,2  ,  3}}  &
\tikztableausmall{{1,1  ,  1,  2},{2,2},  {3}}  &
\tikztableausmall{{1,1,  1,  2},{2,  2,3}}
\tikztableausmall{{1,1,  1,  3},{2,  2,2}}  \\  \QQ'_{331}  =  &
\fS_{331} &+ q\fS_{34} & + q\fS_{421} &+ (q^2+q)\fS_{43}
\end{array}
$$
$$
\begin{array}{ccc}
&
\tikztableausmall{{1,1,  1,2,2},{2},  {3}}&
\tikztableausmall{{1,1,  1,2  ,2},{2  ,  3}}
\tikztableausmall{{1,1  ,  1,  2,  3},{2,2}}\\
&+q^2\fS_{511} &+ (q^3+q^2)\fS_{52} 
\end{array}
$$
$$
\begin{array}{ccc}
&
\tikztableausmall{{1,1,  1,  2,2,2},{3}}
\tikztableausmall{{1,1,  1,  2,2,3},{2}}  &
\tikztableausmall{{1,1,1,2,2,2,3}}\\
&+ (q^4+q^3)\fS_{61} &+
q^5\fS_{7}
\end{array}
$$
\end{Example}

\begin{Remark}
The conjecture has been checked for partitions of size $n \leq 11$. The statement is not true for compositions. The first such example is $\alpha = [1,1,3]$:
\[ \QQ'_{113} = \fS_{113} + q\fS_{122} + q^2\fS_{131} + q^2\fS_{14} +
q^2\fS_{212} + (q^3+q^2-q)\fS_{221}\] \[+ (q^4+q^3)\fS_{23} +
q^2\fS_{311} + (q^5+q^4+q^3-q^2)\fS_{32} + (q^6+q^5)\fS_{41} +
q^7\fS_{5} .\]

This example should be compared to the image $\chi(\QQ'_{113}) = Q'_{113}$ that may be calculated either by
applying the forgetful map to the right hand side or from Definition \ref{def:Qpsym}.
\[ Q'_{113} = \left(q^{3} - q\right)s_{221} + q^{4}s_{311} + \left(q^{5} + q^{4} - q^{2}\right)s_{32} 
+ \left(q^{6} + q^{5}\right)s_{41} + q^{7}s_{5}~. \]
\end{Remark}

\begin{Remark}
The immaculate and the dual immaculate bases is now available in the latest version of Sage. The first, third and fifth authors, with the help of Florent Hivert and Nicolas Thi\'ery, have put the non-commutative and quasi-symmetric functions into Sage.
\end{Remark}

\section*{Appendix: Matrices for $n=4$}

In this appendix we show some examples of the transition matrices.  $M(A,B)$ denotes the transition matrix between bases $A$ and $B$. The rows and columns are indexed by compositions of $4$ in lexicographic order (i.e., $[1,1,1,1]$ is the top row and the leftmost column). In the matrix $M(s,\fS^\ast)$, the rows are indexed by partitions of $4$ again in lexicographic order).

\[
M(H,\fS)=
\left(\begin{array}{cccccccc}
1 & 1 & 2 & 1 & 3 & 3 & 3 & 1 \\
0 & 1 & 1 & 1 & 1 & 2 & 2 & 1 \\
0 & 0 & 1 & 1 & 1 & 2 & 2 & 1 \\
0 & 0 & 0 & 1 & 0 & 1 & 1 & 1 \\
0 & 0 & 0 & 0 & 1 & 1 & 2 & 1 \\
0 & 0 & 0 & 0 & 0 & 1 & 1 & 1 \\
0 & 0 & 0 & 0 & 0 & 0 & 1 & 1 \\
0 & 0 & 0 & 0 & 0 & 0 & 0 & 1
\end{array}\right)
\]

\[
M(\fS,H)=
\left(\begin{array}{cccccccc}
1 & -1 & -1 & 1 & -1 & 1 & 1 & -1 \\
0 & 1 & -1 & 0 & 0 & 0 & 0 & 0 \\
0 & 0 & 1 & -1 & -1 & 0 & 1 & 0 \\
0 & 0 & 0 & 1 & 0 & -1 & 0 & 0 \\
0 & 0 & 0 & 0 & 1 & -1 & -1 & 1 \\
0 & 0 & 0 & 0 & 0 & 1 & -1 & 0 \\
0 & 0 & 0 & 0 & 0 & 0 & 1 & -1 \\
0 & 0 & 0 & 0 & 0 & 0 & 0 & 1
\end{array}\right)
\]

\[
M(R,\fS)=
\left(\begin{array}{cccccccc}
1 & 0 & 0 & 0 & 0 & 0 & 0 & 0 \\
0 & 1 & 1 & 0 & 1 & 0 & 0 & 0 \\
0 & 0 & 1 & 0 & 1 & 1 & 0 & 0 \\
0 & 0 & 0 & 1 & 0 & 1 & 1 & 0 \\
0 & 0 & 0 & 0 & 1 & 0 & 0 & 0 \\
0 & 0 & 0 & 0 & 0 & 1 & 1 & 0 \\
0 & 0 & 0 & 0 & 0 & 0 & 1 & 0 \\
0 & 0 & 0 & 0 & 0 & 0 & 0 & 1
\end{array}\right)
\]

\[
M(\fS,R)=
\left(\begin{array}{cccccccc}
1 & 0 & 0 & 0 & 0 & 0 & 0 & 0 \\
0 & 1 & -1 & 0 & 0 & 1 & -1 & 0 \\
0 & 0 & 1 & 0 & -1 & -1 & 1 & 0 \\
0 & 0 & 0 & 1 & 0 & -1 & 0 & 0 \\
0 & 0 & 0 & 0 & 1 & 0 & 0 & 0 \\
0 & 0 & 0 & 0 & 0 & 1 & -1 & 0 \\
0 & 0 & 0 & 0 & 0 & 0 & 1 & 0 \\
0 & 0 & 0 & 0 & 0 & 0 & 0 & 1
\end{array}\right)
\]

\[
M(\QQ',\fS)=
\left(\begin{array}{cccccccc}
1 & q & q^{2} + q & q^{3} & q^{3} + q^{2} + q &
q^{4} + q^{3} + q^{2} & q^{5} + q^{4} + q^{3} & q^{6} \\
0 & 1 & q & q^{2} & q^{2} & q^{3} + q^{2} &
q^{4} + q^{3} & q^{5} \\
0 & 0 & 1 & q & q & q^{2} + q & q^{3} + q^{2}
& q^{4} \\
0 & 0 & 0 & 1 & 0 & q & q^{2} & q^{3} \\
0 & 0 & 0 & 0 & 1 & q & q^{2} + q & q^{3} \\
0 & 0 & 0 & 0 & 0 & 1 & q & q^{2} \\
0 & 0 & 0 & 0 & 0 & 0 & 1 & q \\
0 & 0 & 0 & 0 & 0 & 0 & 0 & 1
\end{array}\right)
\]

\[
M(\fS,\QQ')=
\left(\begin{array}{cccccccc}
1 & - q & - q & q^{2} & - q & q^{2} & q^{2}
& - q^{3} \\
0 & 1 & - q & 0 & 0 & 0 & 0 & 0 \\
0 & 0 & 1 & - q & - q & q^{2} -  q & q^{2} &
0 \\
0 & 0 & 0 & 1 & 0 & - q & 0 & 0 \\
0 & 0 & 0 & 0 & 1 & - q & - q & q^{2} \\
0 & 0 & 0 & 0 & 0 & 1 & - q & 0 \\
0 & 0 & 0 & 0 & 0 & 0 & 1 & - q \\
0 & 0 & 0 & 0 & 0 & 0 & 0 & 1
\end{array}\right)
\]

\[
M(\QQ',H)=
\]

\noindent
\hspace{-1in}
$
\begin{small}\left(\begin{array}{cccccccc}
1 & q - 1 & q^{2} - 1 & q^{3} -  q^{2} -  q + 1 & q^{3}
- 1 & q^{4} -  q^{3} -  q + 1 & q^{5} -  q^{3} -  q^{2} + 1
& q^{6} -  q^{5} -  q^{4} + q^{2} + q - 1 \\
0 & 1 & q - 1 & q^{2} -  q & q^{2} -  q & q^{3} -
q^{2} & q^{4} - 2 q^{2} + q & q^{5} -  q^{4} -  q^{3} + q^{2} \\
0 & 0 & 1 & q - 1 & q - 1 & q^{2} -  q & q^{3} -
2 q + 1 & q^{4} -  q^{3} -  q^{2} + q \\
0 & 0 & 0 & 1 & 0 & q - 1 & q^{2} -  q &
q^{3} -  q^{2} \\
0 & 0 & 0 & 0 & 1 & q - 1 & q^{2} - 1 &
q^{3} -  q^{2} -  q + 1 \\
0 & 0 & 0 & 0 & 0 & 1 & q - 1 & q^{2} -  q
\\
0 & 0 & 0 & 0 & 0 & 0 & 1 & q - 1 \\
0 & 0 & 0 & 0 & 0 & 0 & 0 & 1
\end{array}\right)\end{small}
$

\[
M(H,\QQ')=
\]

\noindent
$
\begin{small}
\left(\begin{array}{cccccccc}
1 & - q + 1 & -2 q + 2 & q^{2} - 2 q + 1 & -3 q + 3
& 3 q^{2} - 6 q + 3 & 3 q^{2} - 6 q + 3 & - q^{3} + 3 q^{2}
- 3 q + 1 \\
0 & 1 & - q + 1 & - q + 1 & - q + 1 & q^{2} - 3 q +
2 & q^{2} - 3 q + 2 & q^{2} - 2 q + 1 \\
0 & 0 & 1 & - q + 1 & - q + 1 & q^{2} - 3 q + 2
& q^{2} - 3 q + 2 & q^{2} - 2 q + 1 \\
0 & 0 & 0 & 1 & 0 & - q + 1 & - q + 1 & - q
+ 1 \\
0 & 0 & 0 & 0 & 1 & - q + 1 & -2 q + 2 &
q^{2} - 2 q + 1 \\
0 & 0 & 0 & 0 & 0 & 1 & - q + 1 & - q + 1 \\
0 & 0 & 0 & 0 & 0 & 0 & 1 & - q + 1 \\
0 & 0 & 0 & 0 & 0 & 0 & 0 & 1
\end{array}\right)\end{small}
$

\[
M(\QQ',R)= \linebreak
\left(\begin{array}{cccccccc}
1 & q & q^{2} & q^{3} & q^{3} & q^{4} & q^{5}
& q^{6} \\
0 & 1 & q - 1 & q^{2} & q^{2} -  q & q^{3} -  q + 1
& q^{4} -  q^{2} + q - 1 & q^{5} \\
0 & 0 & 1 & q & q - 1 & q^{2} - 1 & q^{3} -  q +
1 & q^{4} \\
0 & 0 & 0 & 1 & 0 & q - 1 & q^{2} -  q &
q^{3} \\
0 & 0 & 0 & 0 & 1 & q & q^{2} & q^{3} \\
0 & 0 & 0 & 0 & 0 & 1 & q - 1 & q^{2} \\
0 & 0 & 0 & 0 & 0 & 0 & 1 & q \\
0 & 0 & 0 & 0 & 0 & 0 & 0 & 1
\end{array}\right)
\]

\[
M(R,\QQ')= \linebreak
\left(\begin{array}{cccccccc}
1 & - q & - q & q^{2} & - q & q^{2} & q^{2}
& - q^{3} \\
0 & 1 & - q + 1 & - q & - q + 1 & q^{2} - 2 q &
q^{2} -  q & q^{2} \\
0 & 0 & 1 & - q & - q + 1 & q^{2} - 2 q + 1 &
q^{2} - 2 q & q^{2} \\
0 & 0 & 0 & 1 & 0 & - q + 1 & - q + 1 & - q
\\
0 & 0 & 0 & 0 & 1 & - q & - q & q^{2} \\
0 & 0 & 0 & 0 & 0 & 1 & - q + 1 & - q \\
0 & 0 & 0 & 0 & 0 & 0 & 1 & - q \\
0 & 0 & 0 & 0 & 0 & 0 & 0 & 1
\end{array}\right)
\]

\[
M(s,\fS^*)=
\left(\begin{array}{cccccccc}
1 & 0 & 0 & 0 & 0 & 0 & 0 & 0 \\
0 & 0 & 0 & 0 & 1 & 0 & 0 & 0 \\
0 & 0 & 0 & -1 & 0 & 1 & 0 & 0 \\
0 & 0 & 0 & 0 & 0 & 0 & 1 & 0 \\
0 & 0 & 0 & 0 & 0 & 0 & 0 & 1
\end{array}\right)
\]

\end{document}